\documentclass[10pt,oneside,leqno]{amsart}
\usepackage{amsxtra}
\usepackage{amsopn}
\usepackage{color}
\usepackage{amsmath,amsthm,amssymb}
\usepackage{amscd}
\usepackage{amsfonts}
\usepackage{latexsym}
\usepackage{verbatim}
\usepackage{multirow}

\theoremstyle{plain}
\newtheorem{theorem}{Theorem}[section]

\newtheorem{lemma}[theorem]{Lemma}
\newtheorem{proposition}[theorem]{Proposition}
\newtheorem{corollary}[theorem]{Corollary}
\newtheorem{remark}[theorem]{Remark}

\newtheorem{remark-question}[section]{Remark-Question}

\newcommand\frg{{\mathfrak g}}
\newcommand\frh{{\mathfrak h}}

\newcommand\gc{\frg_\mathbb{C}}
\newcommand\Real{{\mathfrak R}{\frak e}\,} %%%%% Real and Imaginary parts %%%%%
\newcommand\Imag{{\mathfrak I}{\frak m}\,}
\newcommand\nilm{\Gamma\backslash G}

\newcommand\db{{\bar{\partial}}}
\newcommand\zzz{{\!\!\!}}

\begin{document}
\title[]{On the Bott-Chern cohomology and balanced Hermitian nilmanifolds}

%%%\subjclass[2000]{}
%%%\thanks{This work was supported by the Project...}

\author{A. Latorre}
\address[A. Latorre and L. Ugarte]{Departamento de Matem\'aticas\,-\,I.U.M.A.\\
Universidad de Zaragoza\\
Campus Plaza San Francisco\\
50009 Zaragoza, Spain}
\email{560402@unizar.es}
\email{ugarte@unizar.es}

\author{L. Ugarte}

\author{R. Villacampa}
\address[R. Villacampa]{Centro Universitario de la Defensa\,-\,I.U.M.A., Academia General
Mili\-tar, Crta. de Huesca s/n. 50090 Zaragoza, Spain}
\email{raquelvg@unizar.es}

%%%\date{\today}

\maketitle

\begin{abstract}
The Bott-Chern cohomology of 6-dimensional nilmanifolds endowed with invariant complex structure is studied
with special attention to the cases when
balanced or strongly Gauduchon Hermitian metrics exist.
We consider complex invariants introduced by Angella and Tomassini and by Schweitzer,
which are related to the $\partial\db$-lemma condition and defined in terms of the Bott-Chern cohomology,
and show that the vanishing of some of these invariants is not a closed property under holomorphic deformations.
In the balanced case, we determine the spaces that
parametrize deformations in type IIB supergravity described by Tseng and Yau in terms of the Bott-Chern cohomology
group of bidegree (2,2).
\end{abstract}
%%%\maketitle

%%%\tableofcontents

%%%%%%%%%%%%%%%%%%%%%%%%%%%%%%%%%%%
\section{Introduction}\label{intro}
%%%%%%%%%%%%%%%%%%%%%%%%%%%%%%%%%%%

Given a compact complex manifold $M$, the Dolbeault cohomology groups $H^{p,q}_{\db}(M)$, and more generally the terms $E_r^{p,q}(M)$ in the
Fr\"olicher spectral sequence~\cite{Fro}, are well-known finite dimensional invariants of~$M$.
Other complex invariants are given by the Bott-Chern cohomology groups~\cite{BC,Ae}, which we will denote by $H^{p,q}_{\mathrm{BC}}(M)$.
If $M$ satisfies the $\partial\db$-lemma condition~\cite{DGMS} (in particular, if $M$ is compact K\"ahler), i.e. if $\ker \partial\cap \ker \db \cap {\rm im}\, d={\rm im}\, \partial\db$, or equivalently if
any $\partial$-closed, $\db$-closed and $d$-exact complex form on $M$ is $\partial\db$-exact, then $H^{p,q}_{\mathrm{BC}}(M)\cong E_r^{p,q}(M)$.
However, in general the Bott-Chern cohomology groups do not coincide with $E_r^{p,q}(M)$ and provide new invariants
of the compact complex manifold $M$.

In this paper we consider 6-dimensional nilmanifolds $M=\Gamma\backslash G$ endowed with invariant complex structures $J$. Using the classification
of such complex structures given in~\cite{COUV} and the fact that the Bott-Chern cohomology of $(\Gamma\backslash G,J)$
can be reduced to calculation at the level of
the Lie algebra $\frg$ underlying the nilmanifold~\cite{Ang}, in Section~\ref{calcul} we obtain explicit generators of each Bott-Chern cohomology
group $H^{p,q}_{\mathrm{BC}}(\Gamma\backslash G,J)$. As a consequence, in the cases when the complex structure $J$ admits
balanced or strongly Gauduchon Hermitian metrics we show the dimensions of its Bott-Chern cohomology groups.

C-C. Wu proved in~\cite{Wu} that the $\partial\db$-lemma property for compact complex manifolds $M$ is open under holomorphic deformations, however
the deformation limits of compact complex manifolds satisfying the $\partial\db$-lemma remain unclear.
D. Angella and A. Tomassini gave in~\cite{AT} another proof of the openness of the $\partial\db$-lemma property as a consequence of
general inequalities involving the Bott-Chern cohomology of the manifold~\cite{BC}.
If $n$ is the complex dimension of $M$, let $\mathbf{f}_k(M)=\sum_{p+q=k} \left( h_{p,q}^{\rm BC}(M) + h_{n-p,n-q}^{\rm BC}(M) \right) - 2 b_k(M)$, where
$h_{p,q}^{\mathrm{BC}}(M)$ denotes the dimension of $H^{p,q}_{\mathrm{BC}}(M)$ and $b_k(M)$ the $k$-th Betti number of $M$.
They proved that the complex invariants $\mathbf{f}_k(M)$ are non-negative, and
$\mathbf{f}_k(M)=0$ for all $0\leq k\leq 2n$ if and only if $M$ satisfies the $\partial\db$-lemma~\cite{AT}.
Thus, compact complex manifolds satisfying the $\partial\db$-lemma are characterized as those compact complex manifolds for which
$\mathbf{f}_k(M)= 0$ for all $k$.

On the other hand, if $M$ is a compact complex manifold then for any $r\geq 1$ one can consider the complex invariants
$\mathbf{k}_r(M)=h_{1,1}^{\rm BC}(M)+2 \dim E^{0,2}_r(M) - b_2(M)$,
where $E^{0,2}_r(M)$ is the $r$-step $(0,2)$-term of the Fr\"olicher spectral sequence of $M$.
In Proposition~\ref{relation} we extend a result of M. Schweitzer~\cite{Schw} by showing
that $\mathbf{k}_r(M)\geq 0$ for any $r$, and that $\mathbf{k}_1(M)= 0$
(and therefore all $\mathbf{k}_r$ vanish) when $M$ satisfies the $\partial\db$-lemma.

Section~\ref{BC-hol-def} is devoted to the study of the behaviour of the properties ``$\mathbf{f}_k\!=\! 0$'' (for a given $k$)
and ``$\mathbf{k}_1\!=\! 0$'' under holomorphic deformations $M_t$.
Due to the upper semicontinuity of the Bott-Chern numbers $h_{p,q}^{\rm BC}(M_t)$ with respect to $t$~\cite{Schw}, one has that
the properties $\mathbf{f}_k= 0$ and $\mathbf{k}_1= 0$ are always open under holomorphic deformations.
Using the results of Section~\ref{calcul},
we show in Propositions~\ref{no-closed} and~\ref{K-no-closed}
that they are not closed in general, concretely there is a holomorphic family of compact complex manifolds $M_t$,
$t\in \mathbb{C}$ with $|t|<\frac13$, such that $\mathbf{f}_2(M_t)= 0=\mathbf{k}_1(M_t)$ for $t\not=0$, but $\mathbf{f}_2(M_0)=2=\mathbf{k}_1(M_0)$.

In Section~\ref{supergrav} we put special attention to the cases when balanced Hermitian metrics $F$ exist.
L-S. Tseng and S-T. Yau showed in~\cite{TY} that the Bott-Chern cohomology group $H^{2,2}_{\mathrm{BC}}(M)$
arises in the context of type II string theory as it can be used to count a subset of scalar moduli fields in Minkowski compactifications
with RR fluxes in the presence of O5/D5 brane sources. More precisely, the moduli space of solutions given by linearized variations
is parametrized by the space, which we will denote by $\mathcal L^{2,2}(M,J,F)$, consisting of the harmonic forms of the Bott-Chern cohomology
group $H^{2,2}_{\mathrm{BC}}(M,J)$ which are annihilated by $F$.
Using the explicit description of $H^{2,2}_{\mathrm{BC}}(M)$ given in Section~\ref{calcul} and the results
of~\cite{UV2}, we determine the space $\mathcal L^{2,2}(\Gamma\backslash G,J,F)$ for any invariant balanced Hermitian structure $(J,F)$ on a 6-dimensional nilmanifold $M=\Gamma\backslash G$.
%%%Let us denote by $\mathcal L^{2,2}(\Gamma\backslash G,J,F)$ the moduli space of solutions given by linearized variations.
We show that its dimension only depends on the complex structure, in particular, if $M_0$ denotes the Iwasawa manifold then one has that $\dim \mathcal L^{2,2}(M_0,F)=7$ for any $F$. As an application we show that $\dim \mathcal L^{2,2}$ is not stable under small deformations; concretely, there is a holomorphic deformation $M_t$, $t\in \mathbb{C}$ with $|t|<1$, of the Iwasawa manifold $M_0$ admitting balanced structures for each $t$ and such that $\dim \mathcal L^{2,2}(M_t,F)=5$ for $0<|t|<1$ and for any balanced Hermitian metric $F$ on $M_t$ (Proposition~\ref{vari}).

\vskip.2cm

During the preparation of this paper we were informed by Adriano Tomassini that Daniele Angella, Maria Giovanna Franzini
and Federico Alberto Rossi have obtained in \cite{AFR} similar computations which are used to provide a measure of the degree
of non-K\"ahlerianity of 6-dimensional nilmanifolds with invariant complex structure, as well as the relation between
Bott-Chern cohomological properties and existence of pluriclosed metrics.

%%%%%%%%%%%%%%%%%%%%%%%%%%%%%%%%%%%%%%%%%%%%%%%%%%%%%%%%%%%%%%%%%%%%%%%%%%%%%%%%%%%%%%%%%%%%%%%%%%%%%%%%%%%%%%%%%%%%%%%%%%%%%%%%%%%%%%%%%%%%%%%%%%
\section{Invariant complex structures on nilmanifolds, Bott-Chern and Aeppli cohomologies, and special Hermitian metrics}\label{ComplexStructures}
%%%%%%%%%%%%%%%%%%%%%%%%%%%%%%%%%%%%%%%%%%%%%%%%%%%%%%%%%%%%%%%%%%%%%%%%%%%%%%%%%%%%%%%%%%%%%%%%%%%%%%%%%%%%%%%%%%%%%%%%%%%%%%%%%%%%%%%%%%%%%%%%%%

Let $M$ be a compact complex manifold. The Dolbeault cohomology groups $H^{p,q}_{\db}(M)$, and more generally the terms $E_r^{p,q}(M)$ in the
Fr\"olicher spectral sequence~\cite{Fro}, are well-known finite dimensional invariants of the complex manifold $M$.
On the other hand, the
Bott-Chern and Aeppli cohomologies define additional complex invariants of $M$ given, respectively, by~\cite{Ae,BC}
$$
H^{p,q}_{\rm BC}(M)=
{\ker\{d\colon \Omega^{p,q}(M)\longrightarrow \Omega^{p+q+1}(M) \}
\over
{\rm im}\,\{\partial\db\colon \Omega^{p-1,q-1}(M)\longrightarrow \Omega^{p,q}(M) \}},
$$
and
$$
H^{p,q}_{\rm A}(M)=
{\ker\{\partial\db\colon \Omega^{p,q}(M)\longrightarrow \Omega^{p+1,q+1}(M) \}
\over
{\rm im}\,\{\partial\colon \Omega^{p-1,q}(M)\longrightarrow \Omega^{p,q}(M) \}
+
{\rm im}\,\{\db\colon \Omega^{p,q-1}(M)\longrightarrow \Omega^{p,q}(M) \}}.
$$
By the Hodge theory developed in \cite{Schw}, all these complex invariants are also finite dimensional and one has the isomorphisms
$H^{p,q}_{\mathrm{A}}(M)\cong H^{n-q,n-p}_{\mathrm{BC}}(M)$.
Notice that $H^{q,p}_{\mathrm{BC}}(M) \cong H^{p,q}_{\mathrm{BC}}(M)$ by complex conjugation.

For any $r\geq 1$ and for any $p,q$, there are natural maps
$$
H^{p,q}_{\mathrm{BC}}(M) \longrightarrow E_r^{p,q}(M)
\quad\quad\quad
{\mbox{\rm and }}
\quad\quad\quad
E_r^{p,q}(M) \longrightarrow H^{p,q}_{\mathrm{A}}(M).
$$
Recall that $E_1^{p,q}(M)\cong H^{p,q}_{\db}(M)$ and that the terms for $r=\infty$ provide a decomposition of the de Rham cohomology of the manifold, i.e.
$H^k_{\rm dR}(M,\mathbb{C})\cong \oplus_{p+q=k} E_{\infty}^{p,q}(M)$.

From now on we shall denote by $h_{p,q}^{\mathrm{BC}}(M)$
the dimension of the cohomology group $H^{p,q}_{\mathrm{BC}}(M)$.
The Hodge numbers will be denoted simply by $h_{p,q}(M)$ and the Betti numbers by $b_{k}(M)$.

\medskip

Notice that in general the natural maps above are neither injective nor surjective. However, all the maps are isomorphisms if and only if
$M$ satisfies the $\partial\db$-lemma \cite{DGMS}.

Therefore, if the $\partial\db$-lemma is satisfied then the above invariants coincide and in particular
one has the Hodge decomposition $H^k_{\rm dR}(M,\mathbb{C})\cong \oplus_{p+q=k} H^{p,q}_{\db}(M)$, where
$H^{p,q}_{\db}(M)\cong \overline{H^{q,p}_{\db}(M)}$.

Here we will consider nilmanifolds endowed with an invariant complex structure $J$, i.e. compact quotients of
simply-connected nilpotent Lie groups $G$ by a lattice of
maximal rank such that $J$ stems naturally from a ``complex'' structure $J$ on the Lie algebra $\frg$ of $G$.
Many aspects of this complex geometry have been investigated by several authors, for instance classification, complex cohomologies computation
or deformation problems~\cite{AB,ABD,Ang,COUV,CF,CFP,CFGU2,KS,MPPS,R1,R2,Sakane,S}.
It is well-known that a complex nilmanifold does not satisfy the $\partial\db$-lemma, unless it is a torus, because it is never formal \cite{DGMS,Ha}, so the different cohomologies above do not coincide in general.

Salamon gave in~\cite{S} a characterization of the invariant complex structures $J$ as those endomorphisms
$J\colon \frg\longrightarrow \frg$ such that $J^2=-{\rm Id}$ for which there exists a basis $\{\omega^j\}_{j=1}^n$ of
the $i$-eigenspace $\frg^{1,0}$ of the extension of $J$ to $\gc^*=\frg^*\otimes_{\mathbb{R}}\mathbb{C}$ such that $d\omega^1=0$ and
$$
d \omega^{j} \in \mathcal{I}(\omega^1,\ldots,\omega^{j-1}), \quad
\mbox{ for } j=2,\ldots,n ,
$$
where $\mathcal{I}(\omega^1,\ldots,\omega^{j-1})$ is the ideal in
$\bigwedge\phantom{\!}^* \,\gc^*$ generated by
$\{\omega^1,\ldots,\omega^{j-1}\}$. A consequence of this result is the existence of a nowhere vanishing closed $(n,0)$-form $\Psi$
on the nilmanifold, so $h_{n,0}^{\mathrm{BC}}(M,J)=1$ for any invariant $J$. Moreover, for any $p\geq 1$ the $(p,0)$-form
$\omega^{1}\wedge\cdots\wedge\omega^{p}$ is closed and since
$H^{p,0}_{\rm BC}(M)=\ker \{ d\colon \Omega^{p,0}(M)\longrightarrow \Omega^{p+1}(M) \}$
it defines a non-zero cohomology class.
Therefore, for any invariant complex structure $J$ on a nilmanifold $M$ one has
$h_{p,0}^{\mathrm{BC}}(M,J)\geq 1$ for all $p$.

When the subalgebra $\frg_{1,0}$ is abelian, the complex structure $J$ is called \emph{abelian}~\cite{ABD} and the Lie algebra differential $d$
satisfies $d(\frg^{1,0})\subset \bigwedge^{1,1}(\frg^*)$. On the other hand, the complex structures associated to \emph{complex Lie algebras}
satisfy $d(\frg^{1,0})\subset \bigwedge^{2,0}(\frg^*)$ and we will refer to them as \emph{complex-parallelizable} structures.
Both abelian and complex-parallelizable structures are particular classes of \emph{nilpotent} complex structures~\cite{CFGU2}
for which there is a basis $\{\omega^j\}_{j=1}^n$
for~$\frg^{1,0}$ satisfying $d\omega^1=0$ and
$$
d \omega^j \in \bigwedge\phantom{\!\!}^2 \,\langle
\omega^1,\ldots,\omega^{j-1},
\omega^{\overline{1}},\ldots,\omega^{\overline{j-1}} \rangle, \quad
\mbox{ for } j=2,\ldots,n .
$$

Concerning the calculation of Bott-Chern cohomology for nilmanifolds with invariant complex structure,
Angella has proved the following Nomizu type~\cite{Nomizu} theorem:

\begin{theorem}\label{Angella}\cite{Ang}
If the natural inclusion
\begin{equation}\label{inclusion}
\left(\bigwedge\phantom{\!\!}^{p,\bullet}(\frg^*),\db\right) \hookrightarrow (\Omega^{p,\bullet}(M),\db)
\end{equation}
induces an isomorphism
\begin{equation}\label{iso-Dolbeault}
\iota\colon H_{\db}^{p,q}(\frg) \longrightarrow H_{\db}^{p,q}(M)
\end{equation}
between the Lie-algebra Dolbeault cohomology of $(\frg,J)$ and the Dolbeault cohomology of $M$, then the
natural map
\begin{equation}\label{iso-BC}
\iota\colon H_{\rm BC}^{p,q}(\frg) \longrightarrow H_{\rm BC}^{p,q}(M)
\end{equation}
between the Lie-algebra Bott-Chern cohomology of $(\frg,J)$ and the Bott-Chern cohomology of $M$
is also an isomorphism.
\end{theorem}

Conditions under which the inclusion \eqref{inclusion} induces an isomorphism \eqref{iso-Dolbeault} can be found in~\cite{CF,CFGU2,R2,Sakane}; for instance, it is always true for abelian and complex-parallelizable structures on nilmanifolds.

\medskip

The existence of special Hermitian metrics on a compact complex manifold $M$ is an important question in complex geometry. Let us denote by $F$ the associated fundamental form.
If $F^{n-1}$ is $\partial \db$-closed then the Hermitian metric is
called \emph{standard} or \emph{Gauduchon}. Recall that by~\cite{Gau} there exists a Gauduchon metric in the conformal class of
any Hermitian metric.
Popovici has introduced and studied the special class of \emph{strongly Gauduchon} (\emph{sG} for short) metrics, defined by the condition
$\partial F^{n-1}=\db\alpha$ for some complex form $\alpha$. A particularly interesting class of sG metrics is the one given by the \emph{balanced} Hermitian metrics, defined by the condition $dF^{n-1}=0$~\cite{Mi}.

It is clear that any sG metric is a Gauduchon metric. Notice that if the
map
\begin{equation}\label{inyect}
H^{n,n-1}_{\db}(M)\longrightarrow H^{n,n-1}_{\rm A}(M)
\end{equation}
is injective (for instance if the $\partial\db$-lemma is satisfied or if $h_{0,1}(M)=0$) then any Gauduchon metric is an sG metric;
in fact, if $\partial \db F^{n-1}=0$ then $\partial F^{n-1}$ defines a class in $H^{n,n-1}_{\db}(M)$
such that $[\partial F^{n-1}]_A=0$ in $H^{n,n-1}_{\rm A}(M)$, so the injectivity of \eqref{inyect}
implies the existence of a complex form $\alpha$ such that $\partial F^{n-1}=\db \alpha$.

Therefore, if \eqref{inyect} is injective then there exists an sG metric in the conformal class of
any Hermitian metric.
Notice that by Serre duality and by the dualities between Aeppli and Bott-Chern cohomologies, one has that the
injectivity of \eqref{inyect} implies
$$
h_{0,1}(M)=\dim H^{n,n-1}_{\db}(M) \leq \dim H^{n,n-1}_{\rm A}(M) = h_{0,1}^{\rm BC}(M).
$$

\begin{corollary}\label{abelian-no-inyect}
Let $M$ be a $2n$-dimensional nilmanifold (not a torus) endowed with an abelian complex structure $J$. Then,
the map \eqref{inyect} is not injective.
\end{corollary}

\begin{proof}
It suffices to show that if $J$ is abelian then $h_{0,1}(M) > h_{0,1}^{\rm BC}(M)$.
By Theorem~\ref{Angella} we have
$$
H^{0,1}_{\db}(M)\cong H^{0,1}_{\db}(\frg)=\{\alpha_{0,1}\in\frg^{0,1} \mid \db \alpha_{0,1}=0 \}
\cong \{\alpha_{1,0}\in\frg^{1,0} \mid \partial \alpha_{1,0}=0 \},$$
and
$$
H^{0,1}_{\rm BC}(M)\cong H^{0,1}_{\rm BC}(\frg)=\{\alpha_{0,1}\in\frg^{0,1} \mid d \alpha_{0,1}=0 \}
\cong \{\alpha_{1,0}\in\frg^{1,0} \mid d \alpha_{1,0}=0 \}.
$$
If $J$ is abelian then $\partial(\frg^{1,0})=0$ and therefore $h_{0,1}(M)=n$. So, if $M$ is not a torus then
$h_{0,1}^{\rm BC}(M)<n$ and \eqref{inyect} is not injective.
\end{proof}

%%%%%%%%%%%%%%%%%%%%%%%%%%%%%%%%%%%%%%%%%%%%%%%%%%%%%%%%%%%%
\section{Bott-Chern cohomology in dimension 6}\label{calcul}
%%%%%%%%%%%%%%%%%%%%%%%%%%%%%%%%%%%%%%%%%%%%%%%%%%%%%%%%%%%%

Let $M=\nilm$ be a $6$-dimensional nilmanifold endowed with an invariant complex structure $J$,
and let $\frg$ be the Lie algebra of $G$.
Rollenske proved in \cite[Section 4.2]{R2} that if $\frg\not\cong\frh_7$ then the natural inclusion
\eqref{inclusion} induces an isomorphism \eqref{iso-Dolbeault}, so by Theorem~\ref{Angella} the computation of the Bott-Chern cohomology
is reduced to the Lie-algebra level.

In \cite{Schw} Schweitzer computed the Bott-Chern cohomology of the Iwasawa manifold and in \cite{Ang} Angella calculated the
Bott-Chern cohomology groups of its small deformations.
Notice that by~\cite[Theorem 2.6]{R1} if $J_0$ is an invariant complex structure for which the inclusion
\eqref{inclusion} induces an isomorphism \eqref{iso-Dolbeault} then small deformations $J_a$ of $J_0$
are again invariant for sufficiently small $a$. Since in this section we compute the Bott-Chern cohomology
of any pair $(\frg,J)$, we cover the Bott-Chern cohomology of any invariant complex structure
and its sufficiently small deformations on any $6$-dimensional nilmanifold with underlying Lie algebra not isomorphic to $\frh_7$.

Recall that a 6-dimensional nilpotent Lie algebra has a complex
structure if and only if it is isomorphic to one of the following
Lie algebras \cite{S}:
$$
\begin{array}{rcl}
\frh_{1} &\!\!=\!\!& (0,0,0,0,0,0),\\[-1pt]
\frh_{2} &\!\!=\!\!& (0,0,0,0,12,34),\\[-1pt]
\frh_{3} &\!\!=\!\!& (0,0,0,0,0,12+34),\\[-1pt]
\frh_{4} &\!\!=\!\!& (0,0,0,0,12,14+23),\\[-1pt]
\frh_{5} &\!\!=\!\!& (0,0,0,0,13+42,14+23),\\[-1pt]
\frh_{6} &\!\!=\!\!& (0,0,0,0,12,13),\\[-1pt]
\frh_{7} &\!\!=\!\!& (0,0,0,12,13,23),\\[-1pt]
\frh_{8} &\!\!=\!\!& (0,0,0,0,0,12),\\[0pt]
\frh_{9} &\!\!=\!\!& (0,0,0,0,12,14+25),
\end{array}
\quad\quad\quad
\begin{array}{rcl}
\frh_{10} &\!\!=\!\!& (0,0,0,12,13,14),\\[-1pt]
\frh_{11} &\!\!=\!\!& (0,0,0,12,13,14+23),\\[-1pt]
\frh_{12} &\!\!=\!\!& (0,0,0,12,13,24),\\[-1pt]
\frh_{13} &\!\!=\!\!& (0,0,0,12,13+14,24),\\[-1pt]
\frh_{14} &\!\!=\!\!& (0,0,0,12,14,13+42),\\[-1pt]
\frh_{15} &\!\!=\!\!& (0,0,0,12,13+42,14+23),\\[-1pt]
\frh_{16} &\!\!=\!\!& (0,0,0,12,14,24),\\[-1pt]
\frh^-_{19} &\!\!=\!\!& (0,0,0,12,23,14-35),\\[0pt]
\frh^+_{26} &\!\!=\!\!& (0,0,12,13,23,14+25).
\end{array}
$$
The Lie algebras $\frh_{2},\ldots,\frh_{8}$ are 2-step nilpotent, whereas $\frh_{9},\ldots,\frh_{16}$ and $\frh^-_{19}$ are 3-step.
Notice that $\frh^+_{26}$ is the only 4-step nilpotent Lie algebra having complex structure.

It is well-known that there are, up to equivalence, two complex-parallelizable structures defined by the
equations
\begin{equation}\label{iwa}
d\omega^1=d\omega^2=0,\quad d\omega^3=\rho\,\omega^{12},
\end{equation}
with $\rho=0$ or $1$, and the Lie algebras are $\frh_1$ (for $\rho=0$) and $\frh_5$ (for $\rho=1$), where the latter case corresponds to the Iwasawa manifold. The remaining complex structures in dimension 6 are parametrized, up to equivalence, by the following three families \cite{COUV}:

\vskip.4cm

\textbf{Family I:}\quad $d\omega^1=d\omega^2=0,\quad d\omega^3=\rho\,\omega^{12} + \omega^{1\bar1} + \lambda\,\omega^{1\bar2} + D\,\omega^{2\bar2}$,

\vskip.2cm

\noindent where $\rho=0$ or $1$, $D\in\mathbb C$ with $\Imag D\geq 0$ and $\lambda\in\mathbb R^{\geq 0}$. The complex structure $J$ is abelian if and only if $\rho=0$. The Lie algebra $\frg$ is 2-step nilpotent with first Betti number $\geq 4$, i.e.
$\frg$ is isomorphic to $\frh_{2},\ldots,\frh_{6}$ or $\frh_{8}$.

\vskip.4cm

\textbf{Family II:} \quad $d\omega^1=0,\quad d\omega^2=\omega^{1\bar1},\quad d\omega^3=\rho\,\omega^{12} + B\,\omega^{1\bar2} + c\,\omega^{2\bar1}$,

\vskip.2cm

\noindent where $\rho=0$ or $1$, $B\in\mathbb C$ and $c\in\mathbb R^{\geq 0}$. Moreover, $(\rho, B, c)\neq (0,0,0)$. The complex structure $J$ is abelian if and only if $\rho=0$. The Lie algebra $\frg$ is isomorphic to $\frh_{7}$ or $\frh_{9},\ldots,\frh_{16}$.

\vskip.4cm

\textbf{Family III:} \quad $d\omega^1=0,\quad d\omega^2=\omega^{13} + \omega^{1\bar3},\quad d\omega^3=\varepsilon\,\omega^{1\bar 1} \pm i(\omega^{1\bar2} -\omega^{2\bar1})$,

\vskip.2cm

\noindent where $\varepsilon=0$ or $1$.  In this case, $J$ is a non-nilpotent complex structure. The corresponding Lie algebras are $\frh^-_{19}$ (for $\varepsilon=0$) and $\frh^+_{26}$ (for $\varepsilon=1$).

\vskip.4cm

Next, we show the cohomology groups $H_{\rm BC}^{p,q}(\frg,J)$ for any $J$ given by the Families I, II and III above. Notice that
small deformations of the Iwasawa manifold belong to Family I. It is clear that in all cases
$$
H^{3,0}_{\mathrm{BC}}=\langle[\omega^{123}]\rangle,\quad\quad H^{3,3}_{\mathrm{BC}}=\langle[\omega^{123\bar1\bar2\bar3}]\rangle.
$$
So, up to conjugation, it suffices to describe the Bott-Chern cohomology groups $H^{p,q}_{\mathrm{BC}}$ for $(p,q)=(1,0),(2,0),(1,1),(2,1),(2,2),(3,1)$ and $(3,2)$. In the Appendix we provide the precise dimension of these groups for each Lie algebra $\frg$ and for each complex structure $J$ on $\frg$, up to equivalence, by using the classification of complex structures given in~\cite{ABD,COUV,UV1}.

In what follows we will use the notation: \quad
$\delta_{\text{expression}}=\begin{cases}1,\quad \text{expression}=0,\\ 0,\quad\text{expression}\neq0.\end{cases}$

\vskip.5cm

%%%\subsection{Cohomology groups of complex structures $J$ in Family I}\label{subsection-famI}
\noindent\textbf{Cohomology groups of complex structures $J$ in Family I:}

\vskip.4cm

\noindent $H^{1,0}_{\mathrm{BC}}=\langle[\omega^1],\,[\omega^2]\rangle$, \quad\quad\quad $H^{2,0}_{\mathrm{BC}}=\langle[\omega^{12}],\,\delta_D\,[\omega^{13}]\rangle$,

\vskip.4cm

\noindent $H^{1,1}_{\mathrm{BC}}=\langle[\omega^{1\bar 1}],\,[\omega^{1\bar 2}],\,[\omega^{2\bar 1}],\,
[\omega^{2\bar 2}],\,\delta_{\rho}\,\delta_{D}\,[\omega^{1\bar 3}+\lambda\,\omega^{2\bar 3}],\,
\delta_{\rho}\,\delta_{D}\,[\omega^{3\bar 1}+\lambda\,\omega^{3\bar 2}],$

\vskip.2cm

\hskip.9cm $\delta_{\rho-1}\delta_{\lambda^2-1-|D|^2-2\Real D}[(D+1)\,\omega^{1\bar 3} + D\lambda\,\omega^{2\bar 3}
- \lambda\,\omega^{3\bar 1} - \bar D (D+1)\,\omega^{3\bar 2}]\rangle$,

\vskip.4cm

\noindent $H^{2,1}_{\mathrm{BC}}=\langle[\omega^{12\bar 1}],\,[\omega^{12\bar 2}],\,[\omega^{13\bar 2}],\,[\omega^{12\bar 3}-\rho\,\omega^{23\bar 2}],\,[\omega^{13\bar 1}-D\,\omega^{23\bar 2}],\,[\omega^{23\bar 1}+\lambda\,\omega^{23\bar 2}], \delta_{\rho}\,\delta_{\lambda}\,\delta_D\,[\omega^{13\bar3}]\rangle$,

\vskip.4cm

\noindent $H^{2,2}_{\mathrm{BC}}=\langle\delta_{2\Real D - \lambda^2 - \rho}[\omega^{12\bar1\bar2}],\,[\omega^{12\bar 1\bar3}],\,[\omega^{12\bar 2\bar3}],\,[\omega^{13\bar1\bar2}],\,[\omega^{23\bar 1\bar2}],$

\vskip.2cm

\hskip.9cm $\begin{cases}[\omega^{13\bar2\bar3}],\,[\omega^{23\bar1\bar3}],\,[\omega^{13\bar1\bar3} - D\,\omega^{23\bar2\bar3}],\,\, \lambda=0,\,D\in\mathbb R,\\[4pt]
[\omega^{13\bar2\bar3}+\omega^{23\bar1\bar3} +\lambda\,\omega^{23\bar2\bar3}],\,[\lambda\,\omega^{13\bar1\bar3}+\bar D\,\omega^{13\bar2\bar3} +D\,\omega^{23\bar1\bar3}],\, \text{other case}\end{cases}\rangle$.

%%%Moreover, $[\omega^{12\bar1\bar2}]=0$ if and only if $2\Real D - \lambda^2-\rho\neq0$.

\vskip.4cm

\noindent $H^{3,1}_{\mathrm{BC}}=\langle[\omega^{123\bar1}],\,[\omega^{123\bar2}],\,\delta_{\rho}\,[\omega^{123\bar3}]\rangle$,
\quad\quad\quad
$H^{3,2}_{\mathrm{BC}}=\langle[\omega^{123\bar1\bar2}],\,[\omega^{123\bar1\bar3}],\,[\omega^{123\bar2\bar3}]\rangle$.

\vskip.5cm

%%%\subsection{Cohomology groups of complex structures $J$ in Family II}\label{subsection-famII}
\noindent\textbf{Cohomology groups of complex structures $J$ in Family II:}

\vskip.4cm

\noindent $H^{1,0}_{\mathrm{BC}}=\langle[\omega^1]\rangle$,\quad\quad\quad $H^{2,0}_{\mathrm{BC}}=\langle[\omega^{12}],\,\delta_c\,[\omega^{13}]\rangle$,

\vskip.4cm

\noindent $H^{1,1}_{\mathrm{BC}}=\langle[\omega^{1\bar1}], [\omega^{1\bar2}], [\omega^{2\bar1}], \delta_{B-\rho}[\omega^{1\bar3}+\rho\,\omega^{2\bar2}],\, \delta_{B-\rho}[\omega^{3\bar1}+\rho\,\omega^{2\bar2}],$

\vskip.2cm

\hskip.9cm $(1-\delta_{B-\rho})[(B-\rho)\omega^{1\bar3}+(|B|^2-\rho)\,\omega^{2\bar2} + (\bar B-\rho)\,\omega^{3\bar1}]\rangle,$

\vskip.4cm

\noindent $H^{2,1}_{\mathrm{BC}}=\langle[\omega^{12\bar1}], [\omega^{12\bar2}],\,[\omega^{13\bar1}],\,  [B\,\omega^{12\bar3}+\rho\,\omega^{23\bar1}],\,[c\,\omega^{12\bar3} + \rho\,\omega^{13\bar2}],
\delta_{\rho-1}\delta_{B-1}\delta_c\,[\omega^{13\bar3} + \omega^{23\bar2}]\rangle$,

\vskip.4cm

\noindent $H^{2,2}_{\mathrm{BC}}=\langle[\omega^{12\bar1\bar3}], [\omega^{12\bar2\bar3}],\,[\omega^{13\bar1\bar2}],\,[\omega^{13\bar1\bar3}],\,[\omega^{23\bar1\bar2}],\,  \delta_{\rho-1}\delta_{B}\delta_c\,[\omega^{13\bar2\bar3}],\,\delta_{\rho-1}\delta_{B}\delta_c\,[\omega^{23\bar1\bar3}]$,

\vskip.2cm

\hskip.9cm $\delta_{|B|^2-c^2}[|B|\,\omega^{13\bar2\bar3}-\bar B\,\omega^{23\bar1\bar3}]\rangle$,

\vskip.4cm

\noindent $H^{3,1}_{\mathrm{BC}}=\langle[\omega^{123\bar1}],\,[\omega^{123\bar2}],\,\delta_{\rho}\,[\omega^{123\bar3}]\rangle$,
\quad\quad\quad $H^{3,2}_{\mathrm{BC}}=\langle[\omega^{123\bar1\bar2}],\,[\omega^{123\bar1\bar3}],\,[\omega^{123\bar2\bar3}]\rangle$.

\vskip.5cm

%%%\subsection{Cohomology groups of complex structures $J$ in Family III}\label{subsection-famIII}
\noindent\textbf{Cohomology groups of complex structures $J$ in Family III:}

\vskip.4cm

\noindent $H^{1,0}_{\mathrm{BC}}=\langle[\omega^1]\rangle$,\quad\quad\quad $H^{2,0}_{\mathrm{BC}}=\langle[\omega^{12}]\rangle$,\quad\quad\quad $H^{1,1}_{\mathrm{BC}}=\langle[\omega^{1\bar1}],\,[\omega^{1\bar2} - \omega^{2\bar1}]\rangle$,

\vskip.4cm

\noindent $H^{2,1}_{\mathrm{BC}}=\langle[\omega^{13\bar1}],\,[\omega^{12\bar3}],\,[\omega^{13\bar3} \mp i\, \omega^{12\bar2}]\rangle$,

\vskip.4cm

\noindent $H^{2,2}_{\mathrm{BC}}=\langle[\omega^{12\bar1\bar3}],\,[\omega^{13\bar1\bar2}],\,[\omega^{13\bar2\bar3}+\omega^{23\bar1\bar3}],\, \delta_{\varepsilon}\, [\omega^{23\bar2\bar3}]\rangle$,

\vskip.4cm

\noindent $H^{3,1}_{\mathrm{BC}}=\langle[\omega^{123\bar1}],\, [\omega^{123\bar3}]\rangle$, \quad\quad\quad $H^{3,2}_{\mathrm{BC}}=\langle[\omega^{123\bar1\bar2}],\,[\omega^{123\bar2\bar3}]\rangle$.

%%%Moreover, $$[\omega^{12\bar1\bar2}]=[\omega^{13\bar1\bar3}]=[\omega^{123\bar1\bar3}]=0.$$

\vskip.3cm

The above families together with the Bott-Chern cohomology of the Iwasawa manifold \cite{Schw} cover all the Bott-Chern cohomology groups $H^{p,q}_{\mathrm{BC}}(\frg,J)$ for any 6-dimensional $\frg$ and any $J$.

\vskip.3cm

If $M=\nilm$ is a $6$-dimensional nilmanifold endowed with an invariant complex structure $J$ admitting balanced or sG metrics then
the Lie algebra $\frg$ of $G$ must be isomorphic to $\frh_1,\ldots,\frh_6$ and $\frh_{19}^-$ \cite{COUV}, so in particular
either $J$ is complex-parallelizable or it belongs to Families I or III.
Using the classification results of~\cite{COUV}, in Table~1 we show the complex structures $J$ in Family I, up to equivalence, on $\frh_2,\ldots,\frh_6$ that admit balanced Hermitian metrics together with the dimension of their Bott-Chern cohomology groups.
Notice that the remaining cases admitting balanced metrics are, apart from the trivial case $\frh_1$, the Iwasawa manifold
and any complex structure on $\frh_{19}^-$, for which the Bott-Chern cohomology is obtained by taking $\varepsilon=0$
in Family III above.

\vskip.2cm

We finish this section with an application to the existence of sG metrics on certain 6-nilmanifolds. By Proposition~\ref{abelian-no-inyect} we know that if $J$ is abelian then \eqref{inyect} is not injective.
In contrast, we have:

\begin{proposition}\label{2step-inyect}
Let $M$ be a $6$-dimensional $2$-step nilmanifold endowed with an invariant complex structure $J$ belonging to Family I.
If $J$ is nilpotent but not abelian, then the map \eqref{inyect} is injective.
\end{proposition}

\begin{proof}
First notice that the Lie algebras underlying $M$ are $\frh_2,\frh_4,\frh_5,\frh_6$, because any $J$
on $\frh_1,\frh_3$ and $\frh_8$ is abelian.
Since $\frg\not\cong\frh_7$, we have
$H^{3,2}_{\db}(M)\cong H^{3,2}_{\db}(\frg)$
and
$H^{3,2}_{\rm A}(M)=H^{3,2}_{\rm A}(\frg)$.
A direct calculation from the equations in Family I with $\rho=1$ shows that
$
H^{3,2}_{\db}(\frg)=\langle [\omega^{123\bar1\bar3}],\,[\omega^{123\bar2\bar3}] \rangle
$
and
$
H^{3,2}_{\rm A}(\frg)=\langle [\omega^{123\bar1\bar3}],\,[\omega^{123\bar2\bar3}] \rangle$,
so the natural map $H^{3,2}_{\db}(\frg)\longrightarrow H^{3,2}_{\rm A}(\frg)$ is injective.
\end{proof}

This result explains why on a nilmanifold with underlying Lie algebra isomorphic to $\frh_2,\frh_4,\frh_5$ or $\frh_6$, any invariant $J$-Hermitian metric with respect to a non-abelian nilpotent $J$ is sG~\cite[Proposition 7.3 and Remark 7.4]{COUV}. In fact, any invariant Hermitian metric is Gauduchon, so the injectivity of \eqref{inyect} implies that it is automatically sG.
The classification of non-abelian nilpotent complex structures on $\frh_2,\ldots,\frh_6$ and their Bott-Chern cohomology is given in the Appendix.
Notice that with respect to an abelian complex structure, a metric is sG if and only if it is balanced \cite{COUV}, so the classification of abelian complex structures admitting sG is already given in Table~1.

\vskip.5cm

%%%%%%%%%%%%%%%%%%%%%%%%%%%%%  Table 1  %%%%%%%%%%%%%%%%%%%%%%%%%%%%%%%%%%

%%%\noindent
{\small
\renewcommand{\arraystretch}{1.3}
\begin{tabular}{|c|c|c|c|c|c|c|c|c|c|c|c|}
\hline
&\multicolumn{4}{|c|}{$J$ in Family I admitting balanced metric} & \multicolumn{7}{|c|}{Bott-Chern cohomology of $J$}\\
\hline
$\frg$ &$\rho$ & $\lambda$ &\multicolumn{2}{|c|}{$D=x+iy$} & $h_{1,0}^{\mathrm{BC}}$ & $h_{2,0}^{\mathrm{BC}}$ & $h_{1,1}^{\mathrm{BC}}$ & $h_{2,1}^{\mathrm{BC}}$ & $h_{3,1}^{\mathrm{BC}}$ & $h_{2,2}^{\mathrm{BC}}$ & $h_{3,2}^{\mathrm{BC}}$\\
\hline\hline

\multirow{4}{*}{$\frh_2$} &\multirow{4}{*}{1} & \multirow{4}{*}{1} & \multirow{4}{*}{$y\!>\!0$} & $x\!=\!-1\pm\sqrt{1-y^2}$ & \multirow{4}{*}{2} & \multirow{4}{*}{1} & \multirow{2}{*}{5} & \multirow{4}{*}{6} & \multirow{4}{*}{2} & \multirow{4}{*}{6} &\multirow{4}{*}{3}\\

&& & & $x \!<\! \frac14-y^2$ & & & & & & & \\ \cline{5-5}\cline{8-8}

&& & & $x\neq -1\pm\sqrt{1-y^2}$ & & & \multirow{2}{*}{4} & & & & \\

&& & & $x \!<\! \frac14-y^2$ & & & & & & & \\ \hline\hline

$\frh_3$ & 0 & 0 & \multicolumn{2}{|c|}{$x\!=\!-1$, $y\!=\!0$} & 2 & 1 & 4 & 6 & 3 & 7 & 3 \\ \hline\hline

\multirow{2}{*}{$\frh_4$} &\multirow{2}{*}{1} & \multirow{2}{*}{1} & \multirow{1}{*}{$x\!\not=\! 0\!=\!y$} & $x\!=\!-2$ & \multirow{2}{*}{2} & \multirow{2}{*}{1} & \multirow{1}{*}{5} & \multirow{2}{*}{6} & \multirow{2}{*}{2} & \multirow{2}{*}{6} & \multirow{2}{*}{3}\\ \cline{5-5} \cline{8-8}

& & & $x\!<\!\frac14$ & $x\!\not=\!-2$ & & & 4 & & & & \\  \hline\hline

&\multirow{2}{*}{0} & \multirow{2}{*}{1} & \multirow{2}{*}{$y\!=\!0$} & $x\!=\!0$ & \multirow{2}{*}{} & \multirow{1}{*}{2} & \multirow{1}{*}{6} & \multirow{2}{*}{} & \multirow{2}{*}{3} & \multirow{2}{*}{6} & \multirow{2}{*}{}\\ \cline{5-5} \cline{7-7} \cline{8-8}

& & & & $0\!<\!x\!<\!\frac14$ & & \multirow{2}{*}{1} & \multirow{10}{*}{4} & & & & \\  \cline{2-5} \cline{10-11}

%%%%%%-IWASAWA-%%%%%%      $\frh_5$ &\multicolumn{4}{|c|}{Iwasawa} & 2 & 3 & 4 & 6 & 2 & 8 & 3\\ \cline{2-12}

&  & 0 & $y\!=\!0$ & $-\frac14\!<\!x\!<\!0$ &  & & &  &  & 7 & \\ \cline{3-5} \cline{7-7} \cline{11-11}

 &  & \multirow{2}{*}{$0\!<\!\lambda^2\!<\!\frac{1}{2}$} & \multirow{2}{*}{$x\!=\!0$} & $y\!=\!0$ & \multirow{2}{*}{} & \multirow{1}{*}{2} & & \multirow{2}{*}{} &  & & \multirow{2}{*}{}\\ \cline{5-5} \cline{7-7}

& & & & $0\!<\!y\!<\!\frac{\lambda^2}{2}$ &  & 1 & & & & & \\  \cline{3-5} \cline{7-7}

$\frh_5$ &  & \multirow{2}{*}{$\frac{1}{2}\leq\lambda^2\!<\!1$} & \multirow{2}{*}{$x\!=\!0$} & $y\!=\!0$ & \multirow{2}{*}{2} & \multirow{1}{*}{2} & & \multirow{2}{*}{6} & \multirow{2}{*}{2} & \multirow{2}{*}{6} & \multirow{2}{*}{3}\\ \cline{5-5} \cline{7-7}

& 1 & & & $0\!<\!y\!<\!\frac{1-\lambda^2}{2}$ &  & 1 & & & & & \\  \cline{3-5} \cline{7-7}

&  & \multirow{2}{*}{$1\!<\!\lambda^2\leq 5$} & \multirow{2}{*}{$x\!=\!0$} & $y\!=\!0$ & \multirow{2}{*}{} & \multirow{1}{*}{2} & & \multirow{2}{*}{} & & & \multirow{2}{*}{}\\ \cline{5-5} \cline{7-7}

& & & & $0\!<\!y\!<\!\frac{\lambda^2-1}{2}$ &  & 1 & & & & & \\  \cline{3-5} \cline{7-7}

&  & \multirow{3}{*}{$\lambda^2\!>\!5$} & \multirow{3}{*}{$x\!=\!0$} & $y\!=\!0$ & \multirow{3}{*}{} & \multirow{1}{*}{2} & & \multirow{3}{*}{} & & & \multirow{3}{*}{}\\ \cline{5-5} \cline{7-7}

& & & & $0\!<\!y\!<\!\frac{\lambda^2-1}{2},\quad y\neq\sqrt{\lambda^2-1}$ & & \multirow{2}{*}{1} & & & & & \\ \cline{5-5}  \cline{8-8}

& & & & $y\!=\!\sqrt{\lambda^2-1}\!<\!\frac{\lambda^2-1}{2}$ &  &  & 5 & & & & \\  \hline\hline

$\frh_6$ & 1 & 1 & \multicolumn{2}{|c|}{$x\!=\!0\!=\!y$} & 2 & 2 & 5 & 6 & 2 & 6 & 3 \\ \hline
\end{tabular}
}

\vskip.1cm

\centerline{\textbf{Table 1.} Classification of complex structures on $\frh_2,\ldots,\frh_6$ admitting}

\centerline{balanced Hermitian metric and their Bott-Chern cohomology}

%%%%%%%%%%%%%%%%%%%%%%%%%%%%%%%%%%%%%%%%%%%%%%%%%%%%%%%%%%%%%%%%%%%%%%%%%%%%%%
\section{Bott-Chern cohomology and holomorphic deformations}\label{BC-hol-def}
%%%%%%%%%%%%%%%%%%%%%%%%%%%%%%%%%%%%%%%%%%%%%%%%%%%%%%%%%%%%%%%%%%%%%%%%%%%%%%

In this section we use the explicit description of the Bott-Chern cohomology groups obtained in Section~\ref{calcul}
to show some aspects of their behaviour under holomorphic deformation.

Let $\Delta$ be an open disc around the origin in $\mathbb{C}$. Following \cite[Definition 1.12]{Pop2},
a given property $\mathcal{P}$ of a compact complex manifold is said to be \emph{open}
under holomorphic deformations if for every holomorphic family of compact complex manifolds
$(M,J_a)_{a\in \Delta}$ and for every $a_0\in \Delta$ the following implication holds:

\medskip

$(M,J_{a_0})$ has property $\mathcal{P}$ $\Longrightarrow$ $(M,J_a)$ has property $\mathcal{P}$ for all $a\in\Delta$ sufficiently close to $a_0$.

\medskip

A given property $\mathcal{P}$ of a compact complex manifold is said to be \emph{closed} under holomorphic deformations if for every holomorphic family of compact complex manifolds $(M,J_a)_{a\in \Delta}$ and for every $a_0\in \Delta$ the following implication holds:

\medskip

$(M,J_{a})$ has property $\mathcal{P}$ for all $a\in \Delta\backslash \{a_0\}$ $\Longrightarrow$ $(M,J_{a_0})$ has property $\mathcal{P}$.

\medskip

Concerning the existence of special Hermitian metrics and holomorphic deformations, Alessandrini and Bassanelli proved in \cite{AB}
that the balanced property of compact complex manifolds is not deformation open.
In contrast to the balanced case, Popovici has proved in \cite{Pop1} that the sG property of compact complex manifolds is open under
holomorphic deformations. However, in \cite{COUV} it is shown that the sG property and the balanced property of compact complex manifolds are not closed under holomorphic deformations
(see \cite{Pop2} for a discussion on deformation openness and closedness of various classes
of compact complex manifolds).

On the other hand, Popovici has proved that the existence of sG metric is guaranteed under strong additional conditions, concretely:

\begin{proposition}\cite[Proposition 4.1]{Pop0}\label{Popovici}
Let $M_a$ be a complex analytic family of compact complex manifolds.
If the $\partial \db$-lemma holds on $M_a$ for every $a\in \Delta\backslash \{0\}$,
then $M_0$ has an sG metric.
\end{proposition}

An interesting problem is if the conclusion in the above proposition holds under conditions weaker than the $\partial \db$-lemma.

Recently, Angella and Tomassini have proved \cite{AT} that for a compact complex manifold $M$
$$
\sum_{p+q=k} \left( h_{p,q}^{\mathrm{BC}}(M) + h_{n-p,n-q}^{\mathrm{BC}}(M) \right) \geq 2 b_k(M),\quad\quad 0\leq k\leq 2n,
$$
where $b_k(M)$ denotes the $k$-th Betti number of $M$, and that all the inequalities are equalities if and only if $M$ satisfies the $\partial \db$-lemma.

Let us denote by $\mathbf{f}_k(M)$ the non-negative integer given by
$$
\mathbf{f}_k(M)=\sum_{p+q=k} \left( h_{p,q}^{\rm BC}(M) + h_{n-p,n-q}^{\rm BC}(M) \right) - 2 b_k(M).
$$

For each $0\leq k\leq n$, we consider the property
$$
\mathcal{F}_k=\{ \mbox{the compact complex manifold satisfies } \mathbf{f}_k=0 \}.
$$
By \cite{AT} the compact complex manifold $M$ satisfies the $\partial \db$-lemma if and only if $M$ has the property $\mathcal{F}_k$ for any $k\leq n$.

For each $k$, the property $\mathcal{F}_k$ is open under holomorphic deformation.
In fact, this follows from \cite{AT} and the fact that the dimensions
$h_{p,q}^{\mathrm{BC}}(M_a)$ are upper-semi-continuous
functions at $a$ \cite{Schw}. Therefore, we have:

\begin{proposition}\label{Fk-cor}
Let $M$ be a compact complex manifold for which the property $\mathcal{F}_k$ holds, and let $M_a$ be a small
deformation of $M$.
Then, for sufficiently small $a$ the manifold $M_a$ has the property $\mathcal{F}_k$, $h_{p,q}^{\rm BC}(M_a)=h_{p,q}^{\rm BC}(M)$ and
$h_{n-p,n-q}^{\rm BC}(M_a)=h_{n-p,n-q}^{\rm BC}(M)$ for any $(p,q)$ such that $p+q=k$.
\end{proposition}

It is unknown if the $\partial \db$-lemma is closed under holomorphic deformation.
Concerning a given property $\mathcal{F}_k$, in the following proposition we show that it is non closed under holomorphic deformation.
Concretely, using the results of the previous section, for $k=2$ we have:

\begin{proposition}\label{no-closed}
Let $(M,J_0)$ be a compact nilmanifold with underlying Lie algebra~$\frh_4$ endowed with an abelian complex structure $J_0$.
Then, there is a holomorphic family of compact complex manifolds $(M,J_a)_{a\in \Delta}$,
where $\Delta=\{ a\in \mathbb{C}\mid |a|<\frac13 \}$, such that $(M,J_a)$ satisfies the property $\mathcal{F}_2$
for each $a\in \Delta\backslash \{0\}$,
but $(M,J_0)$ does not satisfy $\mathcal{F}_2$.
\end{proposition}

\begin{proof}
There is only one abelian complex structure on $M$ up to isomorphism, whose Bott-Chern cohomology
is given in the table for $\frh_4$ in Section~\ref{apendice}. Since $b_2(M)=8$, one has that $\mathbf{f}_2(M,J_0)=2$.

Using the Kuranishi's method, Maclaughlin, Pedersen, Poon and Salamon proved in \cite{MPPS} that $J_0$ has a locally complete family of deformations consisting entirely of invariant complex structures and obtained the deformation parameter space in terms of invariant forms. Writing the structure equations of $J_0$ as
$$
d\eta^1=d\eta^2=0,\quad d\eta^3=\frac{i}{2} \eta^{1\bar{1}} +\frac12 \eta^{1\bar{2}} +\frac12 \eta^{2\bar{1}},
$$
by~\cite{KS,MPPS} any complex structure sufficiently near to $J_0$ has a basis of $(1,0)$-forms such that
\begin{equation}\label{def-space}
\left\{
\begin{array}{lcl}
\mu^1 \zzz & = &\zzz \eta^1 + \Phi^1_1\, \eta^{\bar{1}} + \Phi^1_2\, \eta^{\bar{2}},\\[4pt]
\mu^2 \zzz & = &\zzz \eta^2 +\Phi^2_1\, \eta^{\bar{1}} + \Phi^2_2\, \eta^{\bar{2}},\\[4pt]
\mu^3 \zzz & = &\zzz \eta^3 +\Phi^3_1\, \eta^{\bar{1}} + \Phi^3_2\, \eta^{\bar{2}} + \Phi^3_3\, \eta^{\bar{3}},
\end{array}
\right.
\end{equation}
where $i(1+\Phi^3_3)\Phi^1_2=(1-\Phi^3_3)(\Phi^1_1-\Phi^2_2)$ and the coefficients $\Phi^i_j$ are sufficiently small. The
complex structure remains abelian if and only if $\Phi^1_2=0$ and $\Phi^1_1=\Phi^2_2$.

We will consider the particular holomorphic deformation $J_a$ given in \cite{COUV}
by shrinking conveniently the radius of the deformation disc.
For each $a\in \mathbb{C}$ such that $|a|<1$, we consider the basis of (1,0)-forms $\{\mu^1,\mu^2,\mu^3\}$ given by
$$\mu^1=\eta^1 +a \eta^{\bar{1}} -i a \eta^{\bar{2}},\quad \mu^2=\eta^2,\quad \mu^3=\eta^3.$$
Notice that this choice corresponds to $\Phi^1_1=a$, $\Phi^1_2=-ia$ and $\Phi^2_1=\Phi^2_2=\Phi^3_1=\Phi^3_2=\Phi^3_3=0$ in
the parameter space \eqref{def-space}.
Following \cite[Theorem 7.9]{COUV}, for any $a\in \mathbb{C}$ such that $0<|a|<1$ the complex structure $J_a$ is nilpotent but not abelian, and there is a (1,0)-basis $\{\tau^1,\tau^2,\tau^3\}$ such that the structure equations for $J_a$ are
\begin{equation}\label{ecus-Ja}
d\tau^1=d\tau^2=0,\quad d\tau^3=\tau^{12} + \tau^{1\bar1} +\frac{1}{|a|}\, \tau^{1\bar2} + \frac{1-|a|^2}{4|a|^2}\,\tau^{2\bar2}.
\end{equation}
Moreover, using \cite[Proposition 3.7]{COUV} there is a basis $\{\omega^1,\omega^2,\omega^3\}$ of $(1,0)$-forms for which \eqref{ecus-Ja} can
be reduced to the normalized structure equations
\begin{equation}\label{ecus-normal-Ja}
d\omega^1=d\omega^2=0,\quad d\omega^3=\omega^{12} + \omega^{1\bar1} + \omega^{1\bar2} + \frac{|a|^2-1}{4|a|^2}\,\omega^{2\bar2},
\end{equation}
that is, the coefficients $(\rho,\lambda,D)$ satisfy $\rho=\lambda=1$ and $D=\frac{|a|^2-1}{4|a|^2}<0$.

Now, let us compute $\mathbf{f}_2$ for any $J_a$ with $a\not=0$. Since
$$
\mathbf{f}_2(M,J_a)=2\, h_{2,0}^{\mathrm{BC}}(M,J_a) + h_{1,1}^{\mathrm{BC}}(M,J_a) + 2\, h_{3,1}^{\mathrm{BC}}(M,J_a) + h_{2,2}^{\mathrm{BC}}(M,J_a) - 2 b_2(M),
$$
from the table for $\frh_4$ given in the Appendix, we have $h_{2,0}^{\mathrm{BC}}(M,J_a)=1$, $h_{3,1}^{\mathrm{BC}}(M,J_a)=2$ and $h_{2,2}^{\mathrm{BC}}(M,J_a)=6$. Moreover, $h_{1,1}^{\mathrm{BC}}(M,J_a)=4$ if and only if $\frac{|a|^2-1}{4|a|^2}\not=-2$, that is, if and only if $|a|\not=\frac13$. Therefore, for $0<|a|<\frac13$ we conclude that $h_{1,1}^{\mathrm{BC}}(M,J_a)=4$ and therefore $\mathbf{f}_2(M,J_a)=0$.
\end{proof}

\begin{remark}\label{dim-changing}
{\rm
Notice that the holomorphic deformation is defined for $|a|<1$ and the dimensions of the Bott-Chern cohomology groups vary
with $a$ as follows:

\smallskip

$h_{1,1}^{\mathrm{BC}}(M,J_a)=4$ for $a$ such that $|a|\not=\frac13$ and $h_{1,1}^{\mathrm{BC}}(M,J_a)=5$
if $|a|=\frac13$;

\smallskip

$h_{3,1}^{\mathrm{BC}}(M,J_0)=3$ and $h_{3,1}^{\mathrm{BC}}(M,J_a)=2$
for any $0<|a|<1$.
}
\end{remark}

The complex manifold $(M,J_0)$ does not admit any sG metric \cite{COUV}, therefore
as a consequence of Proposition~\ref{no-closed} we get:

\begin{corollary}\label{no-Popovici-1}
Let $(M,J_0)$ be a compact nilmanifold with underlying Lie algebra~$\frh_4$ endowed with abelian complex structure $J_0$.
Then, there is a holomorphic family of compact complex manifolds $(M,J_a)_{a\in \Delta}$,
where $\Delta=\{ a\in \mathbb{C}\mid |a|<\frac13 \}$, such that $(M,J_a)$ satisfies property $\mathcal{F}_2$ for each $a\in \Delta\backslash \{0\}$,
but $(M,J_0)$ does not admit any sG metric.
\end{corollary}

This corollary says that Proposition~\ref{Popovici} does not hold in general if we weaken the $\partial\db$-lemma condition,
that is, having a weaker property $\mathcal{F}_k$ is not sufficient to ensure the existence of a sG metric.

\bigskip

For any compact complex manifold $M$, Schweitzer proved in~\cite[Lemma 3.3]{Schw} that
$$
h_{1,1}^{\rm BC}(M)+2\, h_{0,2}(M) \geq b_2(M),
$$
where $h_{0,2}(M)$ denotes the dimension of the Dolbeault cohomology group $H^{0,2}_{\db}(M)$, and moreover, if $M$
is K\"ahler then the equality holds. More generally, one has

\begin{proposition}\label{relation}
If $M$ is a compact complex manifold then for any $r\geq 1$
$$
h_{1,1}^{\rm BC}(M)+2 \dim E^{0,2}_r(M) \geq b_2(M),
$$
where $E^{0,2}_r(M)$ denotes the $r$-step $(0,2)$-term of the Fr\"olicher spectral sequence.
Moreover, if $M$
satisfies the $\partial\db$-lemma then the above inequalities are all equalities.
\end{proposition}

\begin{proof}
For $r=1$ the term $E^{0,2}_1(M)$ is precisely the Dolbeault cohomology group $H^{0,2}_{\db}(M)$. We recall~\cite{CFGU0} that
$$
E_2^{0,2}(M)=\frac{\{\alpha_{0,2}\in \Omega^{0,2}(M)\,|\,\db \alpha_{0,2}=0,\, \partial\alpha_{0,2}=\db\alpha_{1,1}\}}{\db(\Omega^{0,1}(M))},
$$
$$
E_3^{0,2}(M)=\frac{\{\alpha_{0,2}\in \Omega^{0,2}(M)\,|\,\db \alpha_{0,2}=0,\, \partial\alpha_{0,2}=\db\alpha_{1,1},\, \partial\alpha_{1,1}=\db\alpha_{2,0}\}}{\db(\Omega^{0,1}(M))},
$$
and
$$
E_r^{0,2}(M)=\frac{\{\alpha_{0,2}\in \Omega^{0,2}(M)\,|\,\db \alpha_{0,2}=0,\, \partial\alpha_{0,2}=\db\alpha_{1,1},\, \partial\alpha_{1,1}=\db\alpha_{2,0},\, \partial\alpha_{2,0}=0\}}{\db(\Omega^{0,1}(M))},
$$
for any $r\geq 3$.
Let us consider the sequence
$$
0\rightarrow Z^{1,1}(M)
\hookrightarrow H^{1,1}_{\mathrm{BC}}(M) \stackrel{u}{\rightarrow} H^{2}_{\mathrm{dR}}(M,\mathbb{C}) \stackrel{v_r}{\rightarrow}
\overline{E_r^{0,2}(M)}\oplus E_r^{0,2}(M)
\rightarrow {\rm coker}(v_r) \rightarrow 0,
$$
where
$Z^{1,1}(M)=\frac{\ker\{d\colon \Omega^{1,1}(M)\longrightarrow \Omega^{3}(M) \}\, \cap\, d(\Omega^1)}{ \partial\db(\Omega^0(M)) }$,
$u$ is the natural map and $v_r$ is given by
$$
v_r([\alpha=\alpha_{2,0}+\alpha_{1,1}+\alpha_{0,2}]_{\mathrm{dR}})=([\alpha_{2,0}],[\alpha_{0,2}])\in \overline{E_r^{0,2}(M)}\oplus E_r^{0,2}(M).
$$
The above sequence is exact because $\ker v_r \subset {\rm im}\, u$. In fact, if $([\alpha_{2,0}],[\alpha_{0,2}])=(0,0)$ in $\overline{E_r^{0,2}(M)}\oplus E_r^{0,2}(M)$ then $\alpha_{2,0}=\partial \beta_{1,0}$ and $\alpha_{0,2}=\db \beta_{0,1}$ for some (0,1)-forms $\overline{\beta_{1,0}}$ and $\beta_{0,1}$. Therefore, the (1,1)-form $\gamma=\alpha_{1,1}-\db \beta_{1,0}-\partial\beta_{0,1}$ is closed and
$$
\alpha-\gamma=
d(\beta_{1,0}+\beta_{0,1}),
$$
that is, $\gamma$ defines a class in $H^{1,1}_{\mathrm{BC}}(M)$ such that $u([\gamma]_{\mathrm{BC}})=[\alpha]_{\mathrm{dR}}$.

Now, the exactness of the sequence implies
$$
\begin{array}{rcl}
0&\!\!=&\!\!\dim Z^{1,1}(M)-h_{1,1}^{\mathrm{BC}}(M)+b_2(M)-2 \dim E_r^{0,2}(M) + \dim {\rm coker}(v_r)\\[6pt]
&\!\!\geq&\!\! b_2(M)-h_{1,1}^{\mathrm{BC}}(M)-2 \dim E_r^{0,2}(M).
\end{array}
$$

Under the $\partial\db$-lemma the natural map $H^{p,q}_{\mathrm{BC}}(M) \longrightarrow H^{p,q}_{\db}(M)$ is an isomorphism, so $h_{2,0}(M)=h_{2,0}^{\mathrm{BC}}(M)=h_{0,2}^{\mathrm{BC}}(M)=h_{0,2}(M)$ and $h_{1,1}(M)=h_{1,1}^{\mathrm{BC}}(M)$.
Moreover, by~\cite{DGMS} the Fr\"olicher spectral sequence degenerates at the first step,
therefore for any $r$ one has that $\dim E^{0,2}_r(M)=h_{0,2}(M)$ and
$$
b_2(M)=h_{2,0}(M) + h_{1,1}(M) + h_{0,2}(M)= h_{1,1}^{\mathrm{BC}}(M) + 2\, \dim E^{0,2}_r(M).
$$
\end{proof}

Let $M$ be a compact complex manifold. From now on, we denote by $\mathbf{k}_r(M)$ the non-negative integer given by
$$
\mathbf{k}_r(M)=h_{1,1}^{\rm BC}(M)+2\dim E^{0,2}_r(M) - b_2(M).
$$
Therefore, $\mathbf{k}_r(M)$ are complex invariants which vanish if the manifold $M$ satisfies the $\partial\db$-lemma.
Notice that $\mathbf{k}_1(M)\geq \mathbf{k}_2(M)\geq \mathbf{k}_3(M)=\mathbf{k}_r(M)\geq 0$ for any $r\geq 4$.

\begin{remark}\label{ternas-h15}
{\rm
In general $\mathbf{k}_1(M)$, $\mathbf{k}_2(M)$ and $\mathbf{k}_3(M)$ do not coincide.
%%%Notice that $\mathbf{k}_2\not=\mathbf{k}_3$ implies the non-degeneration of the Fr\"olicher spectral sequence at the second step.
For example, let $M$ be a nilmanifold
with underlying Lie algebra $\frh_{15}$, for which $b_2(M)=5$. From the cases for $\rho=1$ on $\frh_{15}$ given in the Appendix and from \cite[Proposition 6.2]{COUV} we get:
\begin{enumerate}
\item[$\bullet$] if $J$ is a complex structure defined by $\rho=B=1$ and $c>2$, then $h_{1,1}^{BC}(M,J)=5$ and
$\dim E_1^{0,2}(M,J)=\dim E_2^{0,2}(M,J)=2>1=\dim E_3^{0,2}(M,J)$, so
$$
(\mathbf{k}_1(M,J), \mathbf{k}_2(M,J),\mathbf{k}_3(M,J))=(4,4,2);
$$
%%%
\item[$\bullet$] if $J'$ is a complex structure defined by $\rho=1\not=|B|$ and $c=0$, then $h_{1,1}^{BC}(M,J')=4$ and
$\dim E_1^{0,2}(M,J')=2>1=\dim E_2^{0,2}(M,J')=\dim E_3^{0,2}(M,J')$, therefore
$$
(\mathbf{k}_1(M,J'), \mathbf{k}_2(M,J'),\mathbf{k}_3(M,J'))=(3,1,1).
$$
\end{enumerate}
}
\end{remark}

There exist compact complex manifolds $M$ having no K\"ahler metric but with $\mathbf{k}_r(M)=0$, the Iwasawa manifold
being an example. More generally:

\begin{proposition}\label{complex-par}
Let $M$ be a compact complex-parallelizable nilmanifold. Then, $\mathbf{k}_r(M)=0$ for all $r$.
\end{proposition}

\begin{proof}
It suffices to prove the result for $r=1$. Let $\frg$ be the Lie algebra underlying $M$.
Since all the cohomology groups involved can be computed at the Lie-algebra level and since $\db(\frg^{1,0})=0$
because the complex
structure is parallelizable, we have that the sequence
$$
0 \rightarrow H^{1,1}_{\rm BC}(\frg) = \ker \left\{ d\colon\!\! \bigwedge\!\! ^{1,1}(\frg^*)\longrightarrow \bigwedge\!\! ^{3}(\frg^*) \right\}
\stackrel{u}{\hookrightarrow} H^2_{dR}(\frg;\mathbb{C})
\stackrel{v_1}{\longrightarrow} \overline{H^{0,2}_{\db}(\frg)} \oplus H^{0,2}_{\db}(\frg)
\rightarrow 0
$$
is a short exact sequence.
Therefore, $h_{1,1}^{\rm BC}(M) - b_2(M) +2\, h_{0,2}(M) =0$.
\end{proof}

Since the vanishing of $\mathbf{k}_1(M)$ implies the vanishing of any other $\mathbf{k}_r(M)$, next we consider the property
$$
\mathcal{K}=\{ \mbox{the compact complex manifold satisfies } \mathbf{k}_1=0 \}.
%%%\mathcal{K}=\{ \mbox{the compact complex manifold } M \mbox{ satisfies } \mathbf{k}(M)=0 \}.
$$
It is clear that any compact complex manifold satisfying the $\partial\db$-lemma has the property $\mathcal{K}$.

By the same argument as for $\mathcal{F}_k$, the property $\mathcal{K}$ is open under holomorphic deformations.
As a consequence of Proposition~\ref{complex-par} we have:

\begin{corollary}\label{complex-par-cor}
Let $M$ be a compact complex-parallelizable nilmanifold and $M_a$ a small deformation of $M$.
Then, $M_a$ has the property $\mathcal{K}$, $h_{1,1}^{\rm BC}(M_a)=h_{1,1}^{\rm BC}(M)$ and
$h_{0,2}(M_a)=h_{0,2}(M)$ for sufficiently small~$a$.
\end{corollary}

%%%\begin{proof}
%%%This follows from the upper-semi-continuity property
%%%of the dimension of the cohomology groups \cite{Schw}.
%%%\end{proof}

\begin{remark}
{\rm
In particular $h_{1,1}^{\rm BC}$ is stable under small deformations of the Iwasawa manifold.
Notice that for such deformations, Angella \cite{Ang} proved that $h_{1,1}^{\rm BC}=4$ and $h_{0,2}=2$.
}
\end{remark}

\begin{proposition}\label{K-no-closed}
The property $\mathcal{K}$ is not closed under holomorphic deformation.
\end{proposition}

\begin{proof}
The proof is similar to that of Proposition~\ref{no-closed}. For the holomorphic deformation given there, counting the dimension
of the Dolbeault cohomology group obtained in \cite[Proposition 6.1]{COUV},
we get that $h_{0,2}(M,J_0)=3$ and $h_{0,2}(M,J_a)=2$ for any $a\not=0$.
Since $h_{1,1}^{\mathrm{BC}}(M,J_a)=4$ for any $a$ such that $|a|<\frac13$ we conclude that
$\mathbf{k}_1(M,J_a)=0$ for $0<|a|<\frac13$, but $\mathbf{k}_1(M,J_0)=2$.
\end{proof}

The deformation of the abelian complex structure $J_0$ on $\frh_4$ allowed us to show that in general the properties $\mathcal{F}_2$ and $\mathcal{K}$
are not closed under holomorphic deformation.
Moreover, combining Propositions~\ref{no-closed} and~\ref{K-no-closed} with the results on balanced metrics and Fr\"olicher sequence in \cite{COUV}
we have

\begin{theorem}\label{no-Popovici-2}
Let $(M,J_0)$ be a compact nilmanifold with underlying Lie algebra~$\frh_4$ endowed with abelian complex structure $J_0$.
Then, there is a holomorphic family of compact complex manifolds $(M,J_a)_{a\in \Delta}$,
where $\Delta=\{ a\in \mathbb{C}\mid |a|<\frac13 \}$, such that $(M,J_a)$ satisfies properties $\mathcal{F}_2$ and $\mathcal{K}$,
admits balanced metric and has degenerate Fr\"olicher sequence for each $a\in \Delta\backslash \{0\}$,
but $(M,J_0)$ does not admit sG metrics, the properties $\mathcal{F}_2$ and $\mathcal{K}$ fail
and its Fr\"olicher's spectral sequence does not degenerate at the first step.
\end{theorem}

%%%%%%%%%%%%%%%%%%%%%%%%%%%%%%%%%%%%%%%%%%%%%%%%%%%%%%%%%%%%%%%%%%%%%%%%%%%%%%%%%%%%%%%%%%%%%%%%%%%
\section{On the moduli space of type IIB supersymmetric solutions on nilmanifolds}\label{supergrav}
%%%%%%%%%%%%%%%%%%%%%%%%%%%%%%%%%%%%%%%%%%%%%%%%%%%%%%%%%%%%%%%%%%%%%%%%%%%%%%%%%%%%%%%%%%%%%%%%%%%

Tseng and Yau showed in \cite{TY} that the Bott-Chern cohomology can be used to count a subset of scalar moduli fields in Minkowski type IIB compactifications with Ramond-Ramond fluxes and branes. It is motivated by the link that exists between the Maxwell equations
and the de Rham cohomology and the similarities between the type IIB Minkowski supergravity equations (see e.g.~\cite{GMPT}) and the Maxwell equations.
Notice that $N=1$ Minkowski vacua of type II string theories on 6-dimensional nilmanifolds are found in~\cite{GMPT2}.

Recall that type II supergravity equations arise from imposing $N=1$ supersymmetry on the product $\mathbb{R}^{3,1}\times M^6$ of the Minkowski
spacetime and a compact 6-dimensional manifold with the conformally warped metric
$$
ds^2=e^{2f} ds^2_{\mathbb{R}^{3,1}}+ds^2_{M^6},
$$
where $e^{2f}$ is the conformal factor.
In particular,                                                                                                                                                                                                                                                                                                                                                                                                                                                                                                                                                                                                                                                                                                                                                                                                                                                                                                                                                                                                                                                                                                                                                                                                                                                                                                                                                                                                                                                                                                                                                                                                                                                                                                                                                                                                                                                                                                                                                                                                                                                                                                                                                                                      type IIB supersymmetric solutions with $O5/D5$ brane sources are required to be complex.
Such solutions have an SU(3) structure which is encoded in a nowhere vanishing decomposable (3,0)-form $\Psi$ and
a balanced Hermitian (1,1)-form $F$ on $M$.
The SU(3) data $(\Psi,F)$ satisfy $F\wedge \Psi=0$ and $i\Psi\wedge\overline\Psi=\frac43 e^{2f}F^3$, together with the differential conditions
\begin{equation}\label{type_IIB_eq}
d\Psi=0,\quad d(F^2/2)=0,\quad dd^c(e^{-2f}\,F)=\rho_B,
\end{equation}
where $d^c=i(\bar\partial-\partial)$ (and therefore, $dd^c=2\,i\,\partial\bar\partial$) and $\rho_B$ is the Poincar\'e dual four-current of complex codimension two holomorphic submanifolds which the $O5/D5$-branes wrap around. Non-trivial background solutions of the differential system above have a non-zero 3-form flux, $H$, given by $H=d^c(e^{-2f}\,F)$.

Suppose that the complex manifold is fixed, i.e. let us fix $\Psi$. If additionally the source current $\rho_B$ and the conformal factor $e^{-2f}$ do not vary, Tseng and Yau proved that the moduli space of solutions for $F$, given by the linearized variation $F\to F+\delta F$, is parametrized by the harmonic $(1,1)$-forms of the Aeppli cohomology which are also primitive, i.e. $\delta F\in \mathcal H^{1,1}_{\mathrm{A}}(M)\cap  \mathcal P^2$. (We recall that a $(1,1)$-form $\beta$ is primitive if $F^2\wedge\beta=0$.)
Using the duality between Bott-Chern and Aeppli cohomologies (that extends also to the harmonic forms \cite{Schw}), the linearized deformation can be computed in terms of the Bott-Chern cohomology, that is to say, $\delta(F^2/2)=F\wedge \delta F$ is parametrized by
$$
\delta(F^2/2)\in\mathcal H^{2,2}_{\mathrm{BC}}(M)\cap (F\wedge \mathcal P^2),
$$
where $\mathcal H^{2,2}_{\mathrm{BC}}(M)$ denotes the harmonic forms of the Bott-Chern cohomology group $H^{2,2}_{\mathrm{BC}}(M)$, i.e. closed $(2,2)$-forms $\alpha$ such that $\partial\bar\partial(\ast \alpha)=0$, $*$ being the Hodge star operator with respect to the Hermitian metric.
Notice that the space $F\wedge \mathcal P^2$ coincides with the kernel of $F\wedge\cdot:\Omega^{2,2}(M)\to \Omega^{3,3}(M)$ because $F\wedge\cdot:\Omega^{1,1}(M)\to \Omega^{2,2}(M)$ is an isomorphism.

\medskip

From now on, given a compact 6-dimensional manifold $M$ endowed with a complex structure $J$, induced by a
nowhere vanishing closed (3,0)-form $\Psi$, and a balanced $J$-Hermitian structure $F$, i.e.
$d\Psi=0=d F^2$, we denote by $\mathcal L^{2,2}(M,J,F)$ the space
%%%of linearized variations, that is
$$
\mathcal L^{2,2}(M,J,F)=\mathcal H^{2,2}_{\mathrm{BC}}(M,J)\cap (F\wedge \mathcal P^2)
=\{\alpha\in\Omega^{2,2}(M)\,|\, d\alpha=\partial\bar\partial(\ast \alpha)=\alpha\wedge F=0\}.
$$
Motivated by the existence of Minkowski type IIB solutions on 6-dimensional nilmanifolds~\cite{GMPT2},
next we describe the space above for any 6-dimensional nilmanifold $M=\Gamma\backslash G$ endowed with an invariant
balanced Hermitian structure $(J,F)$. Notice that in this case $\mathcal L^{2,2}(M,J,F)$ can be obtained at the
level of the Lie algebra $\frg$ underlying $M=\Gamma\backslash G$. By Section~\ref{calcul} the possible Lie algebras are $\frh_2,\ldots,\frh_6$ and
$\frh_{19}^-$. We will use the description of the space of invariant
balanced Hermitian structures given in~\cite{UV2}.

\medskip

%%%\subsection{Linearized variations in Family I}
\subsection{The space $\mathcal L^{2,2}$ for complex structures in Family I}

Let $J$ be a complex structure in Family I on the Lie algebra $\frg=\frh_{i}$, $2\leq i\leq 6$.
According to \cite{UV2}, any balanced Hermitian metric $F$ is given by
\begin{equation}\label{F-familyI}
2F=i(\omega^{1\bar1} + s^2\,\omega^{2\bar2} + t^2\,\omega^{3\bar3}) + u\,\omega^{1\bar2} -\bar u\,\omega^{2\bar1},
\quad\quad s,t\in \mathbb{R}-\{0\},
\quad u\in \mathbb{C},
\end{equation}
where $s^2>|u|^2$ and $s^2+D=i\lambda\bar u$.
Moreover, by \cite[Proposition~7.7]{COUV} the coefficients of
the complex structure must satisfy $2\,\Real D\neq \lambda^2+\rho$ (see also Table 1).

In order to describe the space $\mathcal L^{2,2}(M,J,F)$, we first notice that any $\partial\db (\omega^{j\bar k})$ vanishes except for $\partial\db (\omega^{3\bar3})=(2\,\Real D-\lambda^2-\rho)\,\omega^{12\bar1\bar2}\neq 0$. So using the Hermitian metric corresponding to \eqref{F-familyI} and the
explicit basis of $H^{2,2}_{\mathrm{BC}}$ given in Section~\ref{calcul}, we have
$$
\mathcal H^{2,2}_{\mathrm{BC}}=\langle \omega^{12\bar 1\bar3},\,\omega^{12\bar 2\bar3},\,\omega^{13\bar1\bar2},\,\omega^{23\bar 1\bar2},\, \omega^{13\bar2\bar3},\,\omega^{23\bar1\bar3},\,\omega^{13\bar1\bar3} - D\,\omega^{23\bar2\bar3}\rangle
$$
when $\lambda=0$ and $D\in\mathbb R$ (notice that $D<0$), and
$$
\mathcal H^{2,2}_{\mathrm{BC}}=\langle \omega^{12\bar 1\bar3},\,\omega^{12\bar 2\bar3},\,\omega^{13\bar1\bar2},\,\omega^{23\bar 1\bar2},\,\omega^{13\bar2\bar3}+\omega^{23\bar1\bar3} +\lambda\,\omega^{23\bar2\bar3},\,\lambda\,\omega^{13\bar1\bar3}+\bar D\,\omega^{13\bar2\bar3} +D\,\omega^{23\bar1\bar3}\rangle,
$$
in any other case.

Now, it is straightforward to verify that the forms $\omega^{12\bar 1\bar3}$, $\omega^{12\bar 2\bar3}$, $\omega^{13\bar1\bar2}$
and $\omega^{23\bar 1\bar2}$ are annihilated by $F$.

If $\lambda=0$ and $D<0$, then a direct computation shows that
$$
\omega^{13\bar 2\bar3}\wedge F=\bar u\,\omega^{123\bar1\bar2\bar3},\quad \omega^{23\bar 1\bar3}\wedge F=-u\,\omega^{123\bar1\bar2\bar3},
\quad
(\omega^{13\bar 1\bar3}-D\,\omega^{23\bar2\bar3})\wedge F=-2iD\,\omega^{123\bar1\bar2\bar3}.
$$
We consider two (2,2)-forms $\theta_1$ and $\theta_2$ depending on the vanishing of the metric coefficient $u$ as follows:
$\theta_1=\omega^{13\bar2\bar3}$ and $\theta_2=\omega^{23\bar1\bar3}$ when $u=0$;
otherwise, $\theta_1=u\,\omega^{13\bar2\bar3} + \bar u\,\omega^{23\bar1\bar3}$ and
$\theta_2=2i\,D\,\omega^{23\bar1\bar3} - u\,(\omega^{13\bar1\bar3}-D\omega^{23\bar2\bar3})$.
It is easy to check that the forms $\theta_1$ and $\theta_2$ are harmonic and satisfy $\theta_i\wedge F=0$ for $i=1,2$.

For the remaining complex structures in Family I admitting balanced metric, the coefficient $\lambda$ is nonzero and in addition to $\omega^{12\bar 1\bar3}$, $\omega^{12\bar 2\bar3}$, $\omega^{13\bar1\bar2}$
and $\omega^{23\bar 1\bar2}$, we can consider one more harmonic (2,2)-form belonging to the space $F\wedge \mathcal P^2$, namely
$$
\gamma=(\lambda s^2 - 2\,\Imag (Du))(\omega^{13\bar2\bar3}+\omega^{23\bar1\bar3} +\lambda\,\omega^{23\bar2\bar3}) - (\lambda - 2\,\Imag u)\,(\lambda\,\omega^{13\bar1\bar3}+\bar D\,\omega^{13\bar2\bar3} +D\,\omega^{23\bar1\bar3}).
$$

\begin{lemma}
Let $J$ be a complex structure in Family I such that $\lambda\neq 0$ and admits balanced $J$-Hermitian metrics. Then $\gamma\neq 0$.
\end{lemma}

\begin{proof}
Let us suppose that $\gamma=0$ and $\lambda\neq 0$. Then $\lambda=2\,\Imag u$ and $\lambda s^2=2\,\Imag(Du)$, and therefore \begin{equation}\label{condicion_gamma}s^2\,\Imag u = \Real u\,\Imag D + \Imag u\,\Real D.\end{equation}  On the other hand, taking real and imaginary parts in $s^2+D=i\,\lambda\,\bar u$ one obtains that $s^2+\Real D=\lambda\,\Imag u$ and $\Imag D=\lambda\,\Real u$.  Substituting in~\eqref{condicion_gamma} we conclude that $s^4=|D|^2$ and condition $s^2+\Real D=\lambda\,\Imag u$ reads as $|D|+\Real D=\frac{\lambda^2}{2}$.  But this contradicts~\cite{COUV} because any complex structure in Family I with $\lambda\neq 0$ and admitting balanced $J$-Hermitian metric satisfies $|D|+\Real D\neq\frac{\lambda^2}{2}$ (see also Table~1).
\end{proof}

\begin{proposition}\label{dim_FI}
Let $M=\Gamma\backslash G$ be a nilmanifold endowed with a complex structure $J$ in Family I admitting balanced metric.
Then, for any invariant balanced $J$-Hermitian structure $F$ we have that
$\dim \mathcal L^{2,2}(\Gamma\backslash G,J,F)=6$ if $\lambda=0$, and $\dim \mathcal L^{2,2}(\Gamma\backslash G,J,F)=5$ otherwise.
Furthermore, with respect to the basis $\{\omega^1,\omega^2,\omega^3\}$ given in Section~\ref{calcul} and the balanced structure \eqref{F-familyI}
we have
$$
\mathcal L^{2,2}(\Gamma\backslash G,J,F)=
\begin{cases}\langle\omega^{12\bar 1\bar3},\,\omega^{12\bar 2\bar3},\,\omega^{13\bar1\bar2},\,\omega^{23\bar 1\bar2},\,\theta_1, \theta_2\rangle,\quad \text{ if }\lambda=0,\\[10pt] \langle\omega^{12\bar 1\bar3},\,\omega^{12\bar 2\bar3},\,\omega^{13\bar1\bar2},\,\omega^{23\bar 1\bar2},\,\gamma\rangle,\quad \text{ if }\lambda\neq0.\end{cases}
$$
\end{proposition}

Notice that if $\lambda=0$ then $\frg\cong\frh_3$ or $\frh_5$. By Table 1, for $\frh_3$ we have $\rho=0$, $D=-1$, $s^2=1$ and $u=0$, whereas for $\frh_5$ we have $\rho=1$, $D\in(-1/4,0)$ and $u$ can be zero or not.

\medskip

%%%\subsection{Linearized variations in Family III}
\subsection{The space $\mathcal L^{2,2}$ for complex structures in Family III}

Let $J$ be a complex structure in Family~III on the Lie algebra $\frh_{19}^-$, i.e. $\varepsilon =0$.
In \cite{UV1} it is proved that any balanced Hermitian metric $F$ is given by
\begin{equation}\label{balanced en h19}
2F=i(r^2\,\omega^{1\bar1} + s^2\,\omega^{2\bar2} + t^2\,\omega^{3\bar3})+v\,(\omega^{2\bar3} -\omega^{3\bar2}),
\quad\quad r,s,t\in \mathbb{R}-\{0\},
\end{equation}
where either $(t^2,v)=(1,0)$ or $v=1$ with $s^2t^2>1$.

Since any $\partial\db (\omega^{j\bar k})$ vanishes except for $\partial\db (\omega^{2\bar2})=-2\,\omega^{13\bar1\bar3}$ and $\partial\db (\omega^{3\bar3})=-2\,\omega^{12\bar1\bar2}$, using the Hermitian metric corresponding to \eqref{balanced en h19} and the
explicit basis of $H^{2,2}_{\mathrm{BC}}$ given in Section~\ref{calcul}, we conclude that
$$
\mathcal H^{2,2}_{\mathrm{BC}}=\langle \omega^{12\bar 1\bar3},\,\omega^{13\bar 1\bar2},\,\omega^{13\bar2\bar3}+\omega^{23\bar 1\bar3},\, \omega^{23\bar2\bar3}\rangle.
$$
A direct computation shows that $(\omega^{13\bar 2\bar3}+\omega^{23\bar1\bar3})\wedge F=0$ and
$$
\omega^{12\bar 1\bar3}\wedge F=v\,\omega^{123\bar1\bar2\bar3},\quad \omega^{13\bar 1\bar2}\wedge F=-v\,\omega^{123\bar1\bar2\bar3},
\quad \omega^{23\bar 2\bar3}\wedge F=i\,r^2\,\omega^{123\bar1\bar2\bar3}.
$$

\begin{proposition}
Let $M_{19}^-$ be a nilmanifold with underlying Lie algebra $\frh_{19}^-$ endowed with a complex structure $J$ in Family III.
Then, $\dim \mathcal L^{2,2}(M_{19}^-,J,F)=3$ for any invariant balanced $J$-Hermitian structure~$F$.
Moreover, with respect to the basis $\{\omega^1,\omega^2,\omega^3\}$ given in Section~\ref{calcul}
and the balanced structure~\eqref{balanced en h19} we have
$$
\mathcal L^{2,2}(M_{19}^-,J,F)=
\begin{cases}\langle\omega^{12\bar 1\bar3},\,\omega^{13\bar 1\bar2},\,\omega^{13\bar2\bar3}+\omega^{23\bar 1\bar3}\rangle,\quad \text{ if }v=0,\\[10pt] \langle\omega^{12\bar 1\bar3}+\omega^{13\bar 1\bar2},\,\omega^{13\bar2\bar3}+\omega^{23\bar 1\bar3},\,i\,r^2\,\omega^{12\bar1\bar3}-\omega^{23\bar 2\bar3}\rangle,\quad \text{ if }v=1.\end{cases}
$$
\end{proposition}

\medskip

%%%\subsection{Linearized variations for the Iwasawa manifold}
\subsection{The space $\mathcal L^{2,2}$ for the Iwasawa manifold}

With respect to the usual (1,0)-basis $\{\omega^1,\omega^2,\omega^3\}$, any balanced Hermitian metric on the Iwasawa manifold $M_0$ is equivalent to one in the following family:
\begin{equation}\label{metrics-J0}
2F=i(\omega^{1\bar1} + \omega^{2\bar2} + t^2 \omega^{3\bar3}),\quad\  t\in \mathbb{R}-\{0\}.
\end{equation}
By~\cite{Schw}, or directly by equations~\eqref{iwa} for $\rho=1$, we have
$$
H^{2,2}_{\mathrm{BC}}=\langle[\omega^{12\bar1\bar3}],\,[\omega^{12\bar2\bar3}],\,
[\omega^{13\bar1\bar2}],\,[\omega^{13\bar1\bar3}],\,[\omega^{13\bar2\bar3}],\,[\omega^{23\bar1\bar2}],\,
[\omega^{23\bar1\bar3}],\,[\omega^{23\bar2\bar3}]\rangle.
$$
It is straightforward to see that the representatives above are harmonic $(2,2)$-forms. Moreover, all the generators are annihilated by $F$ except for $\omega^{13\bar1\bar3}\wedge F=\omega^{23\bar2\bar3}\wedge F=\frac{i}{2}\,\omega^{123\bar1\bar2\bar3}$.

\begin{proposition}\label{Iw}
The space $\mathcal L^{2,2}(M_0,F)$ has dimension $7$ for any invariant balanced Hermitian structure $F$ on the Iwasawa manifold $M_0$. Furthermore,
$$
\mathcal L^{2,2}(M_0,F)=
\langle\omega^{12\bar 1\bar3},\,\omega^{12\bar 2\bar3},\,\omega^{13\bar1\bar2},\,\omega^{13\bar 2\bar3},\,\omega^{23\bar 1\bar2},\,\omega^{23\bar 1\bar3},\,\omega^{13\bar1\bar3}-\omega^{23\bar 2\bar3}\rangle.
$$
\end{proposition}

As a consequence of Propositions~\ref{dim_FI} and~\ref{Iw} we get:

\begin{proposition}\label{vari}
The dimension of the space $\mathcal L^{2,2}$ is not stable under small deformations of the Iwasawa manifold.
%%% There is a 1-parameter family of balanced Hermitian structures $(J_a,F_a)$ on the nilmanifold underlying the Iwasawa manifold such that
%%% the dimension of the moduli space of solutions of the equations of type IIB supergravity varies with $a$.
\end{proposition}

\begin{proof}
By Proposition~\ref{Iw}, $\dim \mathcal L^{2,2}(M_0,F)=7$ for any $F$. We will construct a deformation $M_a$, $a\in \mathbb{C}$ with $|a|<1$, of the Iwasawa manifold $M_0$ admitting balanced structures for each $a$ and such that $\dim \mathcal L^{2,2}(M_a,F)=5$ for $0<|a|<1$ and for any balanced $F$ on $M_a$.

Writing the structure equations of $M_0=(M,J_0)$ as
$$
d\eta^1=d\eta^2=0,\quad d\eta^3= \eta^{12},
$$
by~\cite{KS} any complex structure sufficiently near to $J_0$ has a basis of $(1,0)$-forms of the form \eqref{def-space} with
$\Phi^3_3=\Phi^1_2\Phi^2_1-\Phi^1_1\Phi^2_2$.
For each $a\in \mathbb{C}$ such that $|a|<1$, we consider the complex structure $J_a$ corresponding to $\Phi^2_2=a$ and all the remaining $\Phi^i_j=0$ in~\eqref{def-space}, that is,
\begin{equation}\label{base-mus}
\mu^1=\eta^1,\quad \mu^2=\eta^2 +a \eta^{\bar{2}} ,\quad \mu^3=\eta^3.
\end{equation}
Then, for any $a\in \mathbb{C}$ such that $|a|<1$ the structure equations of $J_a$ are
\begin{equation}\label{ecus-Ja-Iwasawa}
d\mu^1=d\mu^2=0,\quad d\mu^3=\frac{1}{1-|a|^2}(\mu^{12} - a \mu^{1\bar2}).
\end{equation}
Now, for each $a\not=0$, let us consider $\omega^1=\mu^1$, $\omega^2=\frac{1}{\bar{a}} \mu^1+\mu^2$ and $\omega^3=(1-|a|^2)\mu^3$.
Then, with respect to the basis of $(1,0)$-forms $\{\omega^1,\omega^2,\omega^3\}$ the equations \eqref{ecus-Ja-Iwasawa}
are reduced to the normalized structure equations
\begin{equation}\label{ecus-normal-Ja-Iwasawa}
d\omega^1=d\omega^2=0,\quad d\omega^3=\omega^{12} + \omega^{1\bar1} -a\, \omega^{1\bar2}.
\end{equation}
Therefore, the coefficients $(\rho,\lambda,D)$ can be supposed to take the values $\rho=1$, $\lambda=|a|$ and $D=0$.

Now, for any $a$ such that $0<|a|<1$, it follows from Table~1 that $J_a$ has balanced Hermitian structures~$F$.  Moreover, by Proposition~\ref{dim_FI}
we have $\dim \mathcal L^{2,2}(M_a,F)=5$ for any such $F$.
\end{proof}

\begin{remark}\label{explicit-curve}
{\rm
Let $M$ be the nilmanifold underlying the Iwasawa manifold $M_0=(M,J_0)$ and consider a balanced Hermitian metric $F_{0,t}$ given by \eqref{metrics-J0}.
Then from the proof of Proposition~\ref{vari} an explicit deformation of the balanced structure $(J_0,F_{0,t})$ along which the
dimension of $\mathcal L^{2,2}$ varies can be obtained as follows. Let us consider the metric
$$
F_{a,t}=\frac{i}{2} \left( \mu^{1\bar1} + \frac{1}{1-|a|^2}\mu^{2\bar2} + t^2 \mu^{3\bar3} \right),
$$
where $\{\mu^1,\mu^2,\mu^3\}$ denotes the basis \eqref{base-mus}.
Then $F_{a,t}$ is a $J_a$-Hermitian metric for any $a\in\Delta=\{a\in \mathbb{C}\mid |a|<0\}$, and from
\eqref{ecus-Ja-Iwasawa} we have that $F_{a,t}^2$ is a closed form, that is, $F_{a,t}$ is a balanced metric on $M$.
Therefore, $\dim \mathcal L^{2,2}(M,J_a,F_{a,t})=5$ for $a\in \Delta-\{0\}$, but $\dim \mathcal L^{2,2}(M,J_0,F_{0,t})=7$.
}
\end{remark}

%%%%%%%%%%%%%%%%%%%%%%%%%%%%%%%%%%
\section{Appendix}\label{apendice}
%%%%%%%%%%%%%%%%%%%%%%%%%%%%%%%%%%

Here we show the dimensions of the Bott-Chern cohomology groups for each pair $(\frg,J)$, except for the ones corresponding to the torus
and the Iwasawa manifold which can be found in~\cite{Schw}. In the following tables, different values of
the parameters correspond to non-isomorphic complex structures (see \cite{COUV} for details).
Since $h_{3,0}^{\mathrm{BC}}=1=h_{3,3}^{\mathrm{BC}}$, by the duality in the Bott-Chern cohomology it suffices to show the dimensions
$h_{p,q}^{\mathrm{BC}}$ for $(p,q)=(1,0),(2,0),(1,1),(2,1),(2,2),(3,1)$ and $(3,2)$.

\smallskip

%%%\subsection{Family I}\label{tablas-famI}

%%%\par\vfill

%%%\null\vfil

%%%\renewcommand{\arraystretch}{1.3}

\noindent
{\small
\renewcommand{\arraystretch}{1.4}
\begin{tabular}{|c|c|c|c|c|c|c|c|c|c|c|c|}
\hline
 & \multicolumn{4}{|c|}{Family I} & \multicolumn{7}{|c|}{Bott-Chern cohomology}\\
\hline
$\frg$ & $\rho$ & $\lambda$ &\multicolumn{2}{|c|}{$D=x+iy$} & $h_{1,0}^{\mathrm{BC}}$ & $h_{2,0}^{\mathrm{BC}}$ & $h_{1,1}^{\mathrm{BC}}$ & $h_{2,1}^{\mathrm{BC}}$ & $h_{3,1}^{\mathrm{BC}}$ & $h_{2,2}^{\mathrm{BC}}$ & $h_{3,2}^{\mathrm{BC}}$\\ \hline\hline

%%%%%%%%%%%%%%%%%%%%%%%%%%%%%  h2  %%%%%%%%%%%%%%%%%%%%%%%%%%%%%%%%%%

\multirow{6}{*}{$\frh_2$} & \multirow{2}{*}{0} & \multirow{2}{*}{0} & \multirow{2}{*}{$y=1$} & $x=0$ &\multirow{2}{*}{2} & \multirow{2}{*}{1} & \multirow{2}{*}{4} & \multirow{2}{*}{6}&\multirow{2}{*}{3} & 7 & \multirow{2}{*}{3} \\ \cline{5-5}\cline{11-11}

& & & & $x\neq 0$ & & & & & & 6 & \\ \cline{2-12}

& \multirow{4}{*}{1} & \multirow{4}{*}{1} & \multirow{4}{*}{$y>0$} & $x=-1\pm\sqrt{1-y^2}$ & \multirow{4}{*}{2} & \multirow{4}{*}{1} & 5 & \multirow{4}{*}{6} & \multirow{4}{*}{2} & \multirow{3}{*}{6} &\multirow{4}{*}{3}\\ \cline{5-5}\cline{8-8}

& & & & $x\neq -1\pm\sqrt{1-y^2}$ & & & \multirow{3}{*}{4} & & & & \\

& & & & $x\neq 1$ & & & & & & & \\ \cline{5-5}\cline{11-11}

& & & & $x=1$ & & & & & & 7 & \\ \hline\hline

%%%%%%%%%%%%%%%%%%%%%%%%%%%%%  h3  %%%%%%%%%%%%%%%%%%%%%%%%%%%%%%%%%%

$\frh_3$ & 0 & 0 & \multicolumn{2}{|c|}{$\pm 1$} & 2 & 1 & 4 & 6 & 3 & 7 & 3\\ \hline\hline

%%%%%%%%%%%%%%%%%%%%%%%%%%%%%  h4  %%%%%%%%%%%%%%%%%%%%%%%%%%%%%%%%%%

\multirow{4}{*}{$\frh_4$} & 0 & 1 & \multicolumn{2}{|c|}{$\frac{1}{4}$} & 2 & 1 & 4 & 6 & 3 & 6 & 3\\ \cline{2-12}

& \multirow{3}{*}{1} & \multirow{3}{*}{1} & \multicolumn{2}{|c|}{$-2$} & \multirow{3}{*}{2} & \multirow{3}{*}{1} & 5 & \multirow{3}{*}{6} & \multirow{3}{*}{2} & \multirow{2}{*}{6} & \multirow{3}{*}{3}\\ \cline{4-5} \cline{8-8}

& & & \multicolumn{2}{|c|}{$D\in\mathbb{R}-\{-2,0,1\}$} & & & \multirow{2}{*}{4} & & & & \\ \cline{4-5} \cline{11-11}

& & & \multicolumn{2}{|c|}{1} & & & & & & 7 & \\ \hline\hline

%%%%%%%%%%%%%%%%%%%%%%%%%%%%%  h5 abeliana  %%%%%%%%%%%%%%%%%%%%%%%%%%%%%%%%%%

\multirow{15}{*}{$\frh_5$} & \multirow{2}{*}{0} & \multirow{2}{*}{1} & \multicolumn{2}{|c|}{0} & \multirow{2}{*}{2} & 2 & 6 & \multirow{2}{*}{6} & \multirow{2}{*}{3} & \multirow{2}{*}{6} & \multirow{2}{*}{3}\\ \cline{4-5}\cline{7-8}

& & & \multicolumn{2}{|c|}{$D\in\left(0,\frac{1}{4}\right)$} & & 1 & 4 & & & & \\ \cline{2-12}

%%%%%%%%%%%%%%%%%%%%%%%%%%%%%  h5 no abeliana  %%%%%%%%%%%%%%%%%%%%%%%%%%%%%%%%%%

& \multirow{13}{*}{1} & \multirow{5}{*}{0} & \multirow{3}{*}{$y=0$} & $x=0$ & & 2 & & & & 7 & \\ \cline{5-5}\cline{7-7}\cline{11-11}

& & & & $x=\frac{1}{2}$ & & \multirow{4}{*}{1} & & & & 8 & \\ \cline{5-5}\cline{11-11}

& & & & $x\neq 0,\frac{1}{2}, \ \ x>-\frac{1}{4}$ & & & & & & \multirow{2}{*}{7} & \\ \cline{4-5}

& & & $0<y^2<\frac{3}{4}$ & $x=\frac{1}{2}$ & & & & & & & \\ \cline{4-5}\cline{11-11}

& & & $y>0$ & $x\neq\frac{1}{2},\ \ x>y^2-\frac{1}{4}$ & & & & & & \multirow{10}{*}{6} & \\ \cline{3-5}\cline{7-7}

& & \multirow{2}{*}{$0<\lambda^2<\frac{1}{2}$} & \multirow{2}{*}{$x=0$} & $y=0$ & & 2 & & & & & \\ \cline{5-5}\cline{7-7}

& & & & $0<y<\frac{\lambda^2}{2}$ & 2 & 1 & 4 & 6 & 2 & & 3 \\ \cline{3-5}\cline{7-7}

& & \multirow{2}{*}{$\frac{1}{2}\leq\lambda^2<1$} & \multirow{2}{*}{$x=0$} & $y=0$ & & 2 & & & & & \\ \cline{5-5}\cline{7-7}

& & & & $0<y<\frac{1-\lambda^2}{2}$ & & 1 & & & & & \\ \cline{3-5}\cline{7-7}

& & \multirow{2}{*}{$1<\lambda^2\leq 5$} & \multirow{2}{*}{$x=0$} & $y=0$ & & 2 & & & & & \\ \cline{5-5}\cline{7-7}

& & & & $0<y<\frac{\lambda^2-1}{2}$ & & 1 & & & & & \\ \cline{3-5}\cline{7-7}

& & \multirow{3}{*}{$\lambda^2>5$} & \multirow{3}{*}{$x=0$} & $y=0$ & & 2 & & & & & \\ \cline{5-5}\cline{7-7}

& & & & $0<y<\frac{\lambda^2-1}{2},\ y\neq\sqrt{\lambda^2-1}$ & & \multirow{2}{*}{1} & & & & & \\ \cline{5-5}\cline{8-8}

& & & & $0<y<\frac{\lambda^2-1}{2},\ y=\sqrt{\lambda^2-1}$ & & & 5 & & & & \\ \hline\hline

%%%%%%%%%%%%%%%%%%%%%%%%%%%%%  h6  %%%%%%%%%%%%%%%%%%%%%%%%%%%%%%%%%%

$\frh_6$ & 1 & 1 & \multicolumn{2}{|c|}{0} & 2 & 2 & 5 & 6 & 2 & 6 & 3 \\ \hline\hline

%%%%%%%%%%%%%%%%%%%%%%%%%%%%%  h8  %%%%%%%%%%%%%%%%%%%%%%%%%%%%%%%%%%

$\frh_8$ & 0 & 0 & \multicolumn{2}{|c|}{0} & 2 & 2 & 6 & 7 & 3 & 8 & 3 \\ \hline
\end{tabular}
}

%%%%%%%%%%%%%%%%%%%%%%%%%%%%%%%%%%%%%%%%%%%%%%%%%%%%%%%%%%%%%%%%%%%%%%%%%%%%%%%%%%%%%%%%%%%%%%%%%%%%%%%%%%%%%%%%

%%%\subsection{Family II}\

%%%\noindent
\hskip.5cm{\small
\renewcommand{\arraystretch}{1.4}
\begin{tabular}{|c|c|c|c|c|c|c|c|c|c|c|}
\hline
 & \multicolumn{3}{|c|}{Family II} & \multicolumn{7}{|c|}{Bott-Chern cohomology}\\
\hline
$\frg$ & $\rho$ & $B$ & $c$ & $h_{1,0}^{\mathrm{BC}}$ & $h_{2,0}^{\mathrm{BC}}$ & $h_{1,1}^{\mathrm{BC}}$ & $h_{2,1}^{\mathrm{BC}}$ & $h_{3,1}^{\mathrm{BC}}$ & $h_{2,2}^{\mathrm{BC}}$ & $h_{3,2}^{\mathrm{BC}}$\\ \hline\hline

%%%%%%%%%%%%%%%%%%%%%%%%%%%%%  h7  %%%%%%%%%%%%%%%%%%%%%%%%%%%%%%%%%%

$\frh_7$ & 1 & 1 & 0 & 1 & 2 & 5 & 6 & 2 & 5 & 3 \\ \hline\hline

%%%%%%%%%%%%%%%%%%%%%%%%%%%%%  h9  %%%%%%%%%%%%%%%%%%%%%%%%%%%%%%%%%%

$\frh_9$ & 0 & 1 & 1 & 1 & 1 & 4 & 5 & 3 & 6 & 3 \\ \hline\hline

%%%%%%%%%%%%%%%%%%%%%%%%%%%%%  h10  %%%%%%%%%%%%%%%%%%%%%%%%%%%%%%%%%%

$\frh_{10}$ & 1 & 0 & 1 & 1 & 1 & 4 & 5 & 2 & 5 & 3 \\ \hline\hline

%%%%%%%%%%%%%%%%%%%%%%%%%%%%%  h11  %%%%%%%%%%%%%%%%%%%%%%%%%%%%%%%%%%

\multirow{2}{*}{$\frh_{11}$} & \multirow{2}{*}{1} & $B\in\mathbb{R}-\{0,\frac{1}{2},1\}$ & $|B-1|$ & \multirow{2}{*}{1} & \multirow{2}{*}{1} & \multirow{2}{*}{4} & \multirow{2}{*}{5} & \multirow{2}{*}{2} & 5 & \multirow{2}{*}{3} \\ \cline{3-4}\cline{10-10}

& & $\frac{1}{2}$ & $\frac{1}{2}$ & & & & & & 6 & \\ \hline\hline

%%%%%%%%%%%%%%%%%%%%%%%%%%%%%  h12  %%%%%%%%%%%%%%%%%%%%%%%%%%%%%%%%%%

\multirow{2}{*}{$\frh_{12}$} & \multirow{2}{*}{1} & $\mathfrak{Re}B\neq\frac{1}{2}, \ \mathfrak{Im}B \neq 0$ & \multirow{2}{*}{$|B-1|$} & \multirow{2}{*}{1} & \multirow{2}{*}{1} & \multirow{2}{*}{4} & \multirow{2}{*}{5} & \multirow{2}{*}{2} & 5 & \multirow{2}{*}{3} \\ \cline{3-3} \cline{10-10}

& & $\mathfrak{Re}B=\frac{1}{2}, \ \mathfrak{Im}B\neq 0$ & & & & & & & 6 & \\ \hline\hline

%%%%%%%%%%%%%%%%%%%%%%%%%%%%%  h13  %%%%%%%%%%%%%%%%%%%%%%%%%%%%%%%%%%

\multirow{7}{*}{$\frh_{13}$} & \multirow{7}{*}{1} & \multirow{2}{*}{1} & $0<c<2, \ c\neq 1$ & \multirow{7}{*}{1} & \multirow{7}{*}{1} & \multirow{2}{*}{5} & \multirow{7}{*}{5} & \multirow{7}{*}{2} & 5 &\multirow{7}{*}{3} \\ \cline{4-4} \cline{10-10}

 & & & 1 & & & & & & 6 & \\ \cline{3-4} \cline{7-7} \cline{10-10}

 & & \multicolumn{2}{|c|}{$B\neq 1,\ c\neq |B|, |B-1|,$} & & & \multirow{5}{*}{4} & & & \multirow{3}{*}{5} & \\

 & & \multicolumn{2}{|c|}{$(c,|B|)\neq (0,1),$} & & & & & & & \\

 & & \multicolumn{2}{|c|}{$c^4-2(|B|^2+1)c^2+(|B|^2-1)^2<0$} &  &  &  &  &  &  & \\ \cline{3-4} \cline{10-10}

 & & \multicolumn{2}{|c|}{$B\neq 1,\ c=|B|>\frac{1}{2},$} & & & & & & \multirow{2}{*}{6} & \\

 & & \multicolumn{2}{|c|}{$|B|\neq |B-1|$} & & & & & & & \\ \hline\hline

%%%%%%%%%%%%%%%%%%%%%%%%%%%%%  h14  %%%%%%%%%%%%%%%%%%%%%%%%%%%%%%%%%%

\multirow{5}{*}{$\frh_{14}$} & \multirow{5}{*}{1} & 1 & 2 & \multirow{5}{*}{1} & \multirow{5}{*}{1} & 5 & \multirow{5}{*}{5} & \multirow{5}{*}{2} & 5 &\multirow{5}{*}{3} \\ \cline{3-4} \cline{7-7} \cline{10-10}

 & & $|B|=\frac{1}{2}$ & $\frac{1}{2}$ & & & \multirow{4}{*}{4} & & & 6 & \\ \cline{3-4} \cline{10-10}

 & & \multicolumn{2}{|c|}{$c\neq |B-1|,$} & & & & & & \multirow{3}{*}{5} & \\

 & & \multicolumn{2}{|c|}{$(c,|B|)\neq (0,1),\, (\frac{1}{2},\frac{1}{2}),\, (2,1),$} & & & & & & & \\

 & & \multicolumn{2}{|c|}{$c^4-2(|B|^2+1)c^2+(|B|^2-1)^2=0$} &  &  &  &  &  &  & \\ \hline\hline

%%%%%%%%%%%%%%%%%%%%%%%%%%%%%  h15 abeliana  %%%%%%%%%%%%%%%%%%%%%%%%%%%%%%%%%%

\multirow{10}{*}{$\frh_{15}$} & \multirow{3}{*}{0} & 0 & 1 & \multirow{3}{*}{1} & \multirow{2}{*}{1} & 5 & \multirow{3}{*}{5} & \multirow{3}{*}{3} & \multirow{3}{*}{5} & \multirow{3}{*}{3} \\ \cline{3-4} \cline{7-7}

 & & \multirow{2}{*}{1} & $c\neq 0,\, 1$ & & & \multirow{2}{*}{4} & & & & \\ \cline{4-4}\cline{6-6}

 & & & 0 & & 2 & & & & & \\ \cline{2-11}

%%%%%%%%%%%%%%%%%%%%%%%%%%%%%  h15 no abeliana  %%%%%%%%%%%%%%%%%%%%%%%%%%%%%%%%%%

 & \multirow{7}{*}{1} & 0 & 0 & \multirow{7}{*}{1} & \multirow{2}{*}{2} & \multirow{2}{*}{4} & \multirow{7}{*}{5} & \multirow{7}{*}{2} & 7 & \multirow{7}{*}{3} \\ \cline{3-4}\cline{10-10}

 & & $|B|\neq 0,1$ & 0 & & & & & & \multirow{2}{*}{5} & \\ \cline{3-4} \cline{6-7}

 & & 1 & $c>2$ & & \multirow{5}{*}{1} & 5 & & & & \\ \cline{3-4} \cline{7-7} \cline{10-10}

 & & $|B|=c$ & $0<c<\frac{1}{2}$ & & & \multirow{4}{*}{4} & & & 6 & \\ \cline{3-4} \cline{10-10}

 & & \multicolumn{2}{|c|}{$c\neq 0,\ |B-1|$,} & & & & & & \multirow{3}{*}{5} & \\

 & & \multicolumn{2}{|c|}{$B\neq 1,\ |B|\neq c$,} & & & & & & & \\

 & & \multicolumn{2}{|c|}{$c^4-2(|B|^2+1)c^2+(|B|^2-1)^2>0$} & & & & & & & \\ \hline\hline

%%%%%%%%%%%%%%%%%%%%%%%%%%%%%  h16  %%%%%%%%%%%%%%%%%%%%%%%%%%%%%%%%%%

$\frh_{16}$ & 1 & $|B|=1, \ B\neq 1$ & 0 & 1 & 2 & 4 & 5 & 2 & 5 & 3 \\ \hline
\end{tabular}
}

\medskip

%%%%%%%%%%%%%%%%%%%%%%%%%%%%%%%%%%%%%%%%%%%%%%%%%%%%%%%%%%%%%%%%%%%%%%%%%%%%%%%%%%%%%%%%%%%%%%%%%%%%%%%%%%%%%%%%

%%%\subsection{Family III}\

%%%\noindent
\hskip2.6cm{\small
\renewcommand{\arraystretch}{1.4}
\begin{tabular}{|c|c|c|c|c|c|c|c|c|}
\hline
 & Family III & \multicolumn{7}{|c|}{Bott-Chern cohomology}\\
\hline
$\frg$ & $\varepsilon$ & $h_{1,0}^{\mathrm{BC}}$ & $h_{2,0}^{\mathrm{BC}}$ & $h_{1,1}^{\mathrm{BC}}$ & $h_{2,1}^{\mathrm{BC}}$ & $h_{3,1}^{\mathrm{BC}}$ & $h_{2,2}^{\mathrm{BC}}$ & $h_{3,2}^{\mathrm{BC}}$\\ \hline\hline

$\mathfrak{h}_{19}^-$ & 0 & 1 & 1 & 2 & 3 & 2 & 4 & 2 \\ \hline\hline

$\mathfrak{h}_{26}^+$ & 1 & 1 & 1 & 2 & 3 & 2 & 3 & 2 \\ \hline
\end{tabular}

\bigskip

\medskip

\noindent {\bf Acknowledgments.}
We would like to thank Adriano Tomassini, Daniele Angella, Maria Giovanna Franzini and Federico Alberto Rossi for
pointing out the preprint \cite{AFR} and for useful comments and remarks.
This work has been partially supported through Project MICINN (Spain) MTM2011-28326-C02-01.

%%%%%%%%%%%%%%%%%%%%%%%%%%%%%%%%%%%%%%%%%%%%%%%%%%%%%%%%%%%%%%%%%%%%%%%%%%%%%%%%%%%%%%%%%%%%%%%%%%%%%%%%%%%%%%%%%%%%%%%
%\newpage

\end{document}